\documentclass[11pt]{article}
\usepackage{amsthm, amsmath, amssymb, amsfonts, url, booktabs, tikz, setspace, fancyhdr, bm}
\usepackage{hyperref}
\usepackage{amsthm}
\usepackage{geometry}
\usepackage{hyperref, enumerate}
\usepackage[shortlabels]{enumitem}
\usepackage[babel]{microtype}
\usepackage[english]{babel}
\usepackage[capitalise]{cleveref}
\usepackage{comment}
\usepackage{bbm}
\usepackage{csquotes}
\usepackage{mathabx}
\usepackage{tikz}
\usepackage{graphicx}
\usepackage{float}
\usepackage{xcolor}
\usetikzlibrary{decorations.pathmorphing}
\tikzset{
  wavy/.style={decorate, decoration={snake, amplitude=0.7mm, segment length=2.2mm}},
}
\usetikzlibrary{calc}

\counterwithin{figure}{section}

\newtheorem{theorem}{Theorem}[section]
\newtheorem{prop}[theorem]{Proposition}
\newtheorem{lemma}[theorem]{Lemma}

\newtheorem{claim}[theorem]{Claim}

\newtheorem{fact}[theorem]{Fact}

\theoremstyle{definition}
\newtheorem{defn}[theorem]{Definition}
\newtheorem*{defn-non}{Definition}

\newtheorem{ques}[theorem]{Question}

\newlist{Case}{enumerate}{2}
\setlist[Case, 1]{%
    label           =   {\bfseries Case \arabic*.},
    labelindent=1em ,labelwidth=1.3cm, labelsep*=1em, leftmargin =!
}
\setlist[Case, 2]{%
    label           =   {\bfseries Subcase \arabic{Casei}.\arabic*.},
    labelindent=-1em ,labelwidth=1.3cm, labelsep*=1em, leftmargin =!
}

\newenvironment{poc}{\begin{proof}[Proof of claim]}{\end{proof}}

\newcommand{\ceil}[1]{\lceil #1\rceil}
\newcommand{\floor}[1]{\lfloor #1\rfloor}

\usepackage{todonotes}

\title{A Tverberg-type problem of Kalai: Two negative answers to questions of Alon and Smorodinsky, and the power of disjointness}

\author{
Wenchong Chen\thanks{School of Mathematical Sciences and LPMC, Nankai University, Tianjin, China. Emails: 2212161@mail.nankai.edu.cn and wangzhou@nankai.edu.cn. Zhouningxin Wang is supported by National Natural Science Foundation of China under Grant 12301444.}
\and
Gennian Ge\thanks{School of Mathematical Sciences, Capital Normal University, Beijing 100048, China. Email: gnge@zju.edu.cn. Gennian Ge is supported by the National Key Research and Development Program of China under Grant 2025YFC3409900, the National Natural Science Foundation of China under Grant 12231014, and Beijing Scholars Program.}
\and 
Yang Shu\thanks{School of Mathematical Sciences, University of Science and Technology of China, Hefei, 230026, China. Email: shuyangyyyy@mail.ustc.edu.cn.}
\and
Zhouningxin Wang\footnotemark[1]
\and
Zixiang Xu\thanks{Extremal Combinatorics and Probability Group (ECOPRO), Institute for Basic Science (IBS), Daejeon, South Korea. Email: zixiangxu@ibs.re.kr. Zixiang Xu is supported by the Institute for Basic Science (IBS-R029-C4).}
}

\date{}

\begin{document}
\maketitle

\begin{abstract}
Let \(f_r(d,s_1,\ldots,s_r)\) denote the least integer \(n\) such that every \(n\)-point set \(P\subseteq\mathbb{R}^d\) admits a partition \(P=P_1\sqcup\cdots\sqcup P_r\) with the property that for any choice of \(s_i\)-convex sets \(C_i\supseteq P_i\) \((i\in[r])\) one necessarily has \(\bigcap_{i=1}^r C_i\neq\emptyset\), where an \(s_i\)-convex set means a union of \(s_i\) convex sets. A recent breakthrough by Alon and Smorodinsky establishes a general upper bound \[ f_r(d,s_1,\dots,s_r) =O\Big(dr^2\log r\cdot \Big(\prod_{i=1}^r s_i\Big)\cdot \log\Big(\prod_{i=1}^r s_i\Big)\Big). \] Specializing to \(r=2\) resolves the problem of Kalai from the 1970s. They further singled out two particularly intriguing questions: whether \(f_{2}(2,s,s)\) can be improved from \(O(s^2\log s)\) to \(O(s)\), and whether there is a polynomial upper bound \(f_r(d,s,\ldots,s)\le \operatorname{Poly}(r,d,s)\).
We answer both in the negative by showing the exponential lower bound 
\[
f_{r}(d,s,\ldots,s)> s^{r}
\]
for any \(r\ge 2, s\ge 1\) and \(d\ge 2r-2\), which matches the upper bound up to a multiplicative \(\log{s}\) factor for sufficiently large \(s\). Our construction combines a scalloped planar configuration with a direct product of regular \(s\)-gon on the high-dimensional torus \((\mathbb{S}^1)^{r-2}\).

Perhaps surprisingly, if we additionally require that within each block the \(s_i\) convex sets are pairwise disjoint, the picture changes markedly. Let \(F_r(d,s_1,\ldots,s_r)\) denote this disjoint-union variant of the extremal function. 
\begin{itemize}
    \item We show that \(F_{2}(2,s,s)=O(s\log s)\) by connecting it to a suitable line-separating function in the plane.
 
    \item We show when \(s\) is large, \(F_r(d,s,\ldots,s)\) can be bounded by \(O_{r,d}\bigg(s^{\big(1-\frac{1}{2^{d}(d+1)}\big)r+1}\bigg)\) and \(O_{d}(r^{3}\log{r}\cdot s^{2d+3})\) respectively. This builds on a novel connection between the geometric obstruction and hypergraph Tur\'{a}n numbers, in particular, a variant of the Erd\H{o}s box problem.
\end{itemize}
\end{abstract}

\section{Introduction}
\subsection{Background}
Tverberg-type intersection phenomena lies at the heart of combinatorial convexity: one seeks structural conditions under which different organized parts of a finite set of points must meet. The classical starting point is \emph{Radon’s theorem}~\cite{1921Radon}, which states that every set of \(d+2\) points in \(\mathbb{R}^d\) can be split into two parts whose convex hulls intersect; the bound is tight. A far-reaching generalization is \emph{Tverberg’s theorem}~\cite{1966Tverberg}: for integers \(d\ge1\) and \(r\ge2\), any set of \((r-1)(d+1)+1\) points in \(\mathbb{R}^d\) admits a partition into \(r\) parts whose convex hulls have a common point. This threshold is also sharp. Tverberg’s result has spawned a vast literature, including colorful, fractional, and topological extensions~\cite{1982DMColorful,1990comb,1992jlms,1994jlms,2014BLMS,2015JEMS,1996JCTB,2003BUThm,2012DCG,1992JCTA}, with broad connections across various fields~\cite{2009ACMTrans,2020SIAMData,2022BAMSSurvey}; see the surveys~\cite{2018BAMSSurvey,2018Handbook} and references therein.

Motivated by Radon’s viewpoint, Kalai in the 1970s~\cite{2025KalaiBolg} advocated a different axis of generalization: instead of requiring each part to contribute a single convex hull, allow each part to be a \emph{union of few convex sets}. This leads to Radon and Tverberg-type questions for unions of convex sets, asking for thresholds that force two or more such unions to meet. The significance of this extension lies in its practical relevance, as many real-world objects are not perfectly convex but can be approximated by unions of convex components.

In this vein, Alon and Smorodinsky made a recent breakthrough using an extended VC-dimension argument tailored to unions of halfspaces: among other results, they solved the two-part (``Radon-type’’) problem for unions of few convex sets and developed an \(r\)-part framework in \(\mathbb{R}^d\) \cite{2025arxivAlonSmo}. After this, Kalai~\cite{2025KalaiBolg} stated that this now looks to him like a very promising direction in discrete geometry.

To systematize the \(r\)-part problems for unions of convex sets, we adopt the following master parameter. For integers \(r\ge 2\) and \(s_1,\dots,s_r\ge 1\), define \(f_r(d,s_1,\ldots,s_r)\) to be the least integer \(n\) such that every \(n\)-point set \(P\subseteq\mathbb{R}^d\) admits a partition \(P=P_1\sqcup\cdots\sqcup P_r\) with the property that for any choice of \(s_i\)-convex sets \(C_i\supseteq P_i\) \((i\in[r])\) one necessarily has \( \bigcap_{i=1}^r C_i\neq\emptyset ,\)
where an \(s_i\)-convex set is a union of \(s_i\) convex sets. Within this framework, Radon's theorem~\cite{1921Radon} is equivalent to \(f_{2}(d,1,1)=d+2\) and Tverberg's theorem~\cite{1966Tverberg} is equivalent to \(f_{r}(d,1,\ldots,1)=(r-1)(d+1)+1.\) This quantity encapsulates the Radon case and its Tverberg-type extensions for unions of convex sets, and it will be the central object of research in this paper.

\subsection{The results and questions of Alon and Smorodinsky}
The earliest finiteness result toward Kalai’s problem is due to B\'{a}r\'{a}ny and Kalai, who showed via a clever Ramsey-theoretic argument that \(f_{2}(d,s_{1},s_{2})\) is always finite, albeit with an enormous bound as a function of \(d,s_{1},s_{2}\), which relies on another quantitative result of Conlon, Fox, Pach, Sudakov and Suk~\cite{2014TRansAMS}, see \cite{2017Kalai}.

A recent breakthrough of Alon and Smorodinsky~\cite{2025arxivAlonSmo} provides the currently best general upper bound for the Tverberg-type intersection problem in the context of unions of convex sets.

\begin{theorem}[\cite{2025arxivAlonSmo}]
\label{thm:AS-general}
For integers \(d\ge 1\), \(r\ge 2\), and \(s_1,\dots,s_r\ge 1\), there exists a constant \(c>0\) such that
\[
f_r(d,s_1,\dots,s_r)
\le
cdr^{2}\log r\Big(\prod_{i=1}^{r}s_i\Big)\log\Big(\prod_{i=1}^{r}s_i\Big)
\Bigr..
\]
\end{theorem}

Specializing Theorem~\ref{thm:AS-general} to the two-part case (\(r=2\)) yields a near-optimal answer to Kalai’s original question. Beyond their main upper bound, Alon and Smorodinsky~\cite{2025arxivAlonSmo} also record a number of refined estimates in special regimes. For example, in the planar case they determined the exact value \(f_{2}(2,s,1)=2s+2\), while for fixed \(d\ge 4\) they obtained \(f_{2}(d,s,1)=\Theta(ds\log s)\). 
They also gave linear lower bounds in the planar symmetric setting such as \(f_{2}(2,s,s)\ge 4s\), and demonstrated superlinear phenomena in the three-dimensional case \(f_{2}(3,s,s)\ge s^{1+o(1)}\).

Building on these results, they highlighted two particularly intriguing questions.
\begin{ques}[\cite{2025arxivAlonSmo}]\label{question:AS}\
\begin{enumerate}
  \item 
  Is \(f_{2}(2,s,s)\) linear in \(s\)? 
  \item 
 Is \(f_r(d,s,\dots,s)\) upper bounded by a polynomial in \(r\), \(d\), and \(s\)?
\end{enumerate}
\end{ques}

\subsection{Our contributions}\label{subsection:Contribution}
We settle both questions of Alon and Smorodinsky in the negative. In the planar case, we present an elementary yet previously unnoticed construction that already forces quadratic growth. 

\begin{theorem}\label{thm:quadraticLB}
For every integer \(s\ge 1\), \(
  f_{2}(2,s,s) > s^{2}.\)
\end{theorem}

Beyond the plane, we show that the same obstruction persists in higher dimensions and for larger numbers of parts. The argument lifts the planar gadget through a careful gluing scheme that arranges many low-dimensional carriers in convex position, yielding the following bound.

\begin{theorem}\label{thm:highdLB}
For all integers \(s\ge 1\), \(r\ge 2\) and \(d\ge 2r-2\),
\[
  f_{r}(d,s,\ldots,s) > s^{r}.
\]
\end{theorem}
For large \(s\), this lower bound matches the general upper bound of Alon and Smorodinsky up to a multiplicative factor of \(\log s\). A refined version actually yields a slightly stronger bound \(f_2(d,s,s) > ds^2\), and correspondingly \(f_r(d,s,\ldots,s)>(d-2r+4)\cdot s^r\) for \(d\ge2r-2\), which indicates the upper bound \(O_{r}(ds^{r}\log{s})\) is very close to optimal when both of \(s\) and \(d\) are large enough. Since the proof of \(f_r(d,s,\ldots,s)>(d-2r+4)\cdot s^r\) might become considerably longer and obscure the elegance of the core idea, to preserve the simplicity of the argument, we will present this refined construction  in~\cref{sec:remarks}.

A subtle but consequential modeling choice lies in whether the $s_i$ convex pieces inside each union are allowed to overlap. While the definition of $f_r(d,s_1,\ldots,s_r)$ does not impose such a restriction, we discover that enforcing pairwise disjointness within each block fundamentally changes the quantitative landscape. 
To formalize this restricted setting, we introduce the following variant.

\begin{defn}\label{def:Fr}
For integers $d\ge1$, $r\ge2$, and $s_1,\ldots,s_r\ge1$, let $F_r(d,s_1,\ldots,s_r)$ be the least integer $n$ such that every $n$-point set $P\subseteq\mathbb{R}^d$ admits a partition
$P=P_1\sqcup\cdots\sqcup P_r$ with the following property: If for each \(i\in [r]\), \(C_{i}\) consists of \(s_{i}\) convex sets \(C_{i,j}\) (\(j\in [s_i]\)) such that $C_{i,j}\cap C_{i,j'}=\emptyset$ for all $j\ne j'$ (that is, \(C_{i}=\bigsqcup_{j=1}^{s_{i}}C_{i,j}\)) and \(P_{i}\subseteq C_{i}\), then $\bigcap_{i=1}^{r} C_i\neq\emptyset$ must hold. 
\end{defn}
By the definitions of two functions, we can see the disjoint-union requirement restricts the admissible containers, hence, for all parameters,
\begin{equation}\label{equation:RelationFf}
  F_r(d,s_1,\ldots,s_r) \le f_r(d,s_1,\ldots,s_r).
\end{equation}
What is perhaps very striking is that, unlike the unrestricted model (Theorem~\ref{thm:quadraticLB}), the disjoint-union model undergoes a near-linear regime already in the planar two-part case. This phenomenon aligns with the intuition that motivated Alon and Smorodinsky's question (Question~\ref{question:AS}(1)).

\begin{theorem}\label{thm:VariantUBPlanar}
    There exists an absolute constant \(c>0\) such that for all \(s\ge 3\), \(F_{2}(2,s,s)\le cs\log{s}.\)
\end{theorem}

The passage from the planar near-linear phenomenon to higher parameters is not a routine lifting. 
Instead, geometric constraints can be encoded into a certain incidence hypergraph. This translation allows us to establish a connection with hypergraph Tur\'{a}n theory, yielding the following shape of bounds in the disjoint-union setting.

\begin{theorem}\label{thm:GeneralUpperbound}
Let $d$ and $r$ be two positive integers with $r\ge d+2$, and let $s$ be a sufficiently large integer relative to $d$ and $r$. Then there exists some constant \(c_{d,r}>0\) such that
\[
 F_r(d,s,\ldots,s)
  \le c_{d,r}\cdot \min\bigg\{s^{\bigl(1-\frac{1}{2^{d}(d+1)}\bigr)r+1}, s^{2d+3}\bigg\}.
\]
\end{theorem}
Regarding Question~\ref{question:AS}(2), Alon and Smorodinsky~\cite{2025arxivAlonSmo} showed that for fixed \(r\), \(f_{r}(d,s,\ldots,s)\le \textup{Poly}(s,d)\); by \eqref{equation:RelationFf} the polynomial bound also holds for \(F_{r}(d,s,\ldots,s)\). In the complementary regime, according to the proof of \cref{thm:GeneralUpperbound}, one can see that for fixed \(d\), \(F_{r}(d,s,\ldots,s)\) is also bounded by \(O_{d}(r^{3}\log{r}\cdot s^{2d+3})\). Moreover, when both \(r\) and \(d\) are fixed, \cref{thm:GeneralUpperbound} improves the bound \(O_{r,d}(s^{r})\) of \cref{thm:AS-general} to \(O_{r,d}\big(s^{(1-\varepsilon_{d})r+1}\big)\) for \(\varepsilon_{d}=\frac{1}{2^{d}(d+1)}\) and \(O_{r,d}\big(s^{2d+3}\big)\). It is not hard to see the bound \(O_{r,d}(s^{2d+3})\) performs much better when \(r\) is larger than \(3d\).

All results can be extended to the asymmetric setting by similar arguments. For instance, one can directly obtain \(f_{r}(d,s_{1},\ldots,s_{r})>\prod_{i=1}^{r}s_{i}\) for any \(d\ge 2r-2\), which also yields that the upper bound in~\cref{thm:AS-general} is near-optimal for large \(s_{1},\ldots,s_{r}\).

\section{Scalloped \(s\)-gon in the plane and \(f_{2}(2,s,s)> s^{2}\)}
In this section, we develop a planar construction and give a complete proof of \cref{thm:quadraticLB}. We begin by fixing the notation and recalling basic geometric facts in \cref{sec:prepa}. In \cref{sec:Viewpoints} we present an intuitive overview meant to be read independently of the proof. We then detail the selection of the \(s^{2}\) points and establish the disjointness property, which together complete the argument.
\subsection{Notation and some geometric facts}\label{sec:prepa}
Let \(\mathbb{R}^n\) be the \(n\)-dimensional Euclidean space with its metric topology. For \(\varepsilon>0\) and \(\bm{x}\in \mathbb{R}^n\), we write
\(B(\bm{x},\varepsilon):=\{\bm{y}\in\mathbb{R}^n:\|\bm{y}-\bm{x}\|<\varepsilon\}\) for the open ball of radius \(\varepsilon\) centered at \(\bm{x}\).

\paragraph{Interior.}
The \emph{interior} of \(S\subseteq\mathbb{R}^n\) is
\[
S^{\circ}\ :=\ \{\,\bm{x}\in S:\ \exists\, \varepsilon>0\ \textup{such that}\ B(\bm{x},\varepsilon)\subseteq S\,\}.
\]
Moreover, if $S$ is a subset of some proper subspace $W\subseteq\mathbb{R}^{n}$, the \emph{relative interior} of $S$ with respect to $W$ is defined as its interior as a subset of $W$. 

\paragraph{Boundary.}
The \emph{boundary} of \(S\subseteq\mathbb{R}^n\) is
\[
\partial S\ :=\ \overline{S}\setminus S^{\circ}
\ =\
\bigl\{\,\bm{x}\in\mathbb{R}^n:\ \forall\, r>0,\ B(\bm{x},\varepsilon)\cap S\neq\emptyset\ \text{and}\ B(\bm{x},\varepsilon)\setminus S\neq\emptyset\,\bigr\},
\]
where \(\overline{S}\) denotes the (topological) \emph{closure} of \(S\) in \(\mathbb{R}^n\).

\medskip
For \(\bm{a},\bm{b},\bm{c}\in\mathbb{R}^n\), we use $\angle(\bm{b}-\bm{a},\bm{b}-\bm{c})$ for the induced angle between vectors $\bm{b}-\bm{a}$ and $\bm{b}-\bm{c}$. Besides, we use $(\bm{b}-\bm{a})\cdot(\bm{b}-\bm{c})$ for the dot product of these two vectors.

Following the notion from~\cite{2021CombConvex}, 
a \emph{closed halfspace} $\mathbb{H}\subseteq\mathbb{R}^n$ is the set $\{\bm{x}\in \mathbb{R}^n:\alpha^T\bm{x}\geq a\}$ for some $\alpha\in\mathbb{R}^n$ and $a\in\mathbb{R}$.
A \emph{hyperplane}
$H\subseteq\mathbb{R}^n$ is the set $\{\bm{x}\in\mathbb{R}^n:\alpha^T\bm{x}=a\}$ for some $\alpha\in\mathbb{R}^n$, $a\in\mathbb{R}$. For points \(\bm{x}_{1},\ldots,\bm{x}_{m}\in\mathbb{R}^{n}\), define its \emph{convex hull} as 
\[
\operatorname{Conv}(\{\bm{x}_{1},\ldots,\bm{x}_{m}\})=\{\bm{x}=\sum_{i=1}^{m}\alpha_{i}\bm{x}_{i}:\alpha_{i}\ge0,\ \sum_{i=1}^{m}\alpha_{i}=1\}.
\]
In particular, \(\operatorname{Conv}(\emptyset)=\emptyset.\) A \emph{convex polytope} \(P\subseteq \mathbb{R}^n\) is the convex hull of finitely many points and a \emph{convex polyhedron} is the intersection of finitely many halfspaces.
For a convex polytope $P=\operatorname{Conv}(\{\bm{x}_1,\bm{x}_2,\ldots,\bm{x}_m\})$, the points $\bm{x}_1, \bm{x}_2,\ldots,\bm{x}_m$ are usually called the \emph{vertices} or \emph{generators} of $P$.
 A set \(X\subseteq \mathbb{R}^n\) is \emph{in convex position} if no point of \(X\) lies in the convex hull of the others, that is,
for every \(\bm{x}\in X\), \(\bm{x}\notin \operatorname{Conv}\bigl(X\setminus\{\bm{x}\}\bigr).\) 
For a point set $V\subseteq \mathbb{R}^n$, 
a point \(\bm{v}\in V\) is an \emph{extreme point} of \(V\) if \(\bm{v}\notin\textup{Conv}(V\setminus\{\bm{v}\})\). We write \(\operatorname{Ext}(V)\) for the set of extreme points of \(V\), then 
\[
\operatorname{Ext}(V)=\bigl\{\bm{v}\in V: \bm{v}\notin \operatorname{Conv}\bigl(V\setminus\{\bm{v}\}\bigr)\bigr\}.
\]
We collect a few elementary geometric facts for later use.
\begin{fact}\label{fact:COnvexPosition}\
    \begin{enumerate}
        \item[\textup{(1)}] Any subset of the unit circle \(\mathbb{S}^{1}\subseteq \mathbb{R}^{2}\) is in convex position.
        \item[\textup{(2)}] For any integer \(d\ge 1\), consider the standard product embedding
\[
\mathbb{T}^{d}:=\underbrace{\mathbb{S}^{1}\times\cdots\times\mathbb{S}^{1}}_{d\text{ times}}
 \subseteq \mathbb{R}^{2d},
\]
where the \(k\)-th copy of \(\mathbb{S}^{1}\) lies in its own coordinate plane. Then every subset of \(\mathbb{T}^{d}\) is in convex position.
\item[\textup{(3)}] If a set \(X\subseteq \mathbb{R}^{n}\) is in convex position, then \(\mathrm{Ext}(X)=X\).
    \end{enumerate}
\end{fact}

\subsection{Geometric viewpoint of the construction}\label{sec:Viewpoints}
To show the lower bound \(f_{r}(d,s_{1},\ldots,s_{r})>f\), it suffices to construct a set of points \(P\subseteq\mathbb{R}^{d}\) of size \(f\) with the following property: for any partition \(P=\bigsqcup_{i=1}^{r}P_{i}\), there exists a family of sets \(C_{1},C_{2},\ldots,C_{r}\) with \(P_{i}\subseteq C_{i}\), where \(C_{i}\) is an \(s_{i}\)-convex set for every \(i\in [r]\), such that \(\bigcap_{i=1}^{r}C_{i}=\emptyset.\) 

Our construction is strikingly simple: take an \(s\)-gon in the plane, replace each side by a short inward-curving circular arc from a circle of extremely large radius to form a \emph{scalloped \(s\)-gon}, see~\cref{fig:Sca10} and~\cref{fig:Sca30}, and then select \(s\) points in the central segment of each arc.

\begin{figure}[htbp]
\centering
\begin{minipage}[t]{0.48\textwidth}
\centering
\begin{tikzpicture}[scale=1.2]
    \foreach \i in {0,...,9} {
        \coordinate (P\i) at (90+36*\i:2);
    }
    
    \foreach \i [evaluate=\i as \next using {int(mod(\i+1,10))}] in {0,...,9} {
        \pgfmathsetmacro{\colorvalue}{\i/9*100}
       
        \path (P\i) -- (P\next) coordinate[midway] (M\i);
        \coordinate (C\i) at ($(M\i)!0.04!(0,0)$); 
        
        \draw[line width=1.5pt, color=red!\colorvalue!yellow!\colorvalue!blue!] 
            (P\i) .. controls (C\i) .. (P\next);
    }
\end{tikzpicture}
\caption{Scalloped \(10\)-gon}
\label{fig:Sca10}
\end{minipage}
\begin{minipage}[t]{0.48\textwidth}
\centering
\begin{tikzpicture}[scale=0.85]
    
    \foreach \i in {0,...,29} {
        \coordinate (P\i) at (90+12*\i:3);
    }
  
    \foreach \i [evaluate=\i as \next using {int(mod(\i+1,30))}] in {0,...,29} {
        \pgfmathsetmacro{\colorvalue}{\i/29*100}
      
        \path (P\i) -- (P\next) coordinate[midway] (M\i);
        \coordinate (C\i) at ($(M\i)!0.02!(0,0)$); 
        
        \draw[line width=1.5pt, color=red!\colorvalue!yellow!\colorvalue!blue!] 
            (P\i) .. controls (C\i) .. (P\next);
    }
\end{tikzpicture}
\caption{Scalloped \(30\)-gon}
\label{fig:Sca30}
\end{minipage}
\end{figure}

Formally, we arrange \(s\) many sufficiently large and congruent disks at equally spaced directions around the origin and, on each disk, select \(s\) points along a tiny boundary arc that faces the origin. This produces \(s^{2}\) points naturally organized into \emph{rows} (points on the same disk) and \emph{columns} (points at the same angular position across different disks). Given an arbitrary bipartition of these \(s^{2}\) points, we form one union of convex sets by taking the convex hull of each row inside the first part, and another union by taking the convex hull of each column inside the second part. It is easy to see that every row-wise convex hull is contained in its corresponding disk. Moreover, the choice of very large radii and extremely short arcs ensures that every column-wise convex hull lies in the supporting halfspaces determined by tangents at those points, and therefore avoids the interiors of all disks. Consequently, the two unions are disjoint for every bipartition, which is the geometric mechanism underlying the desired lower bound. For illustration, we include a schematic one with \(s=6\) in~\cref{figure:66666666}, from which we believe the construction becomes immediately transparent.

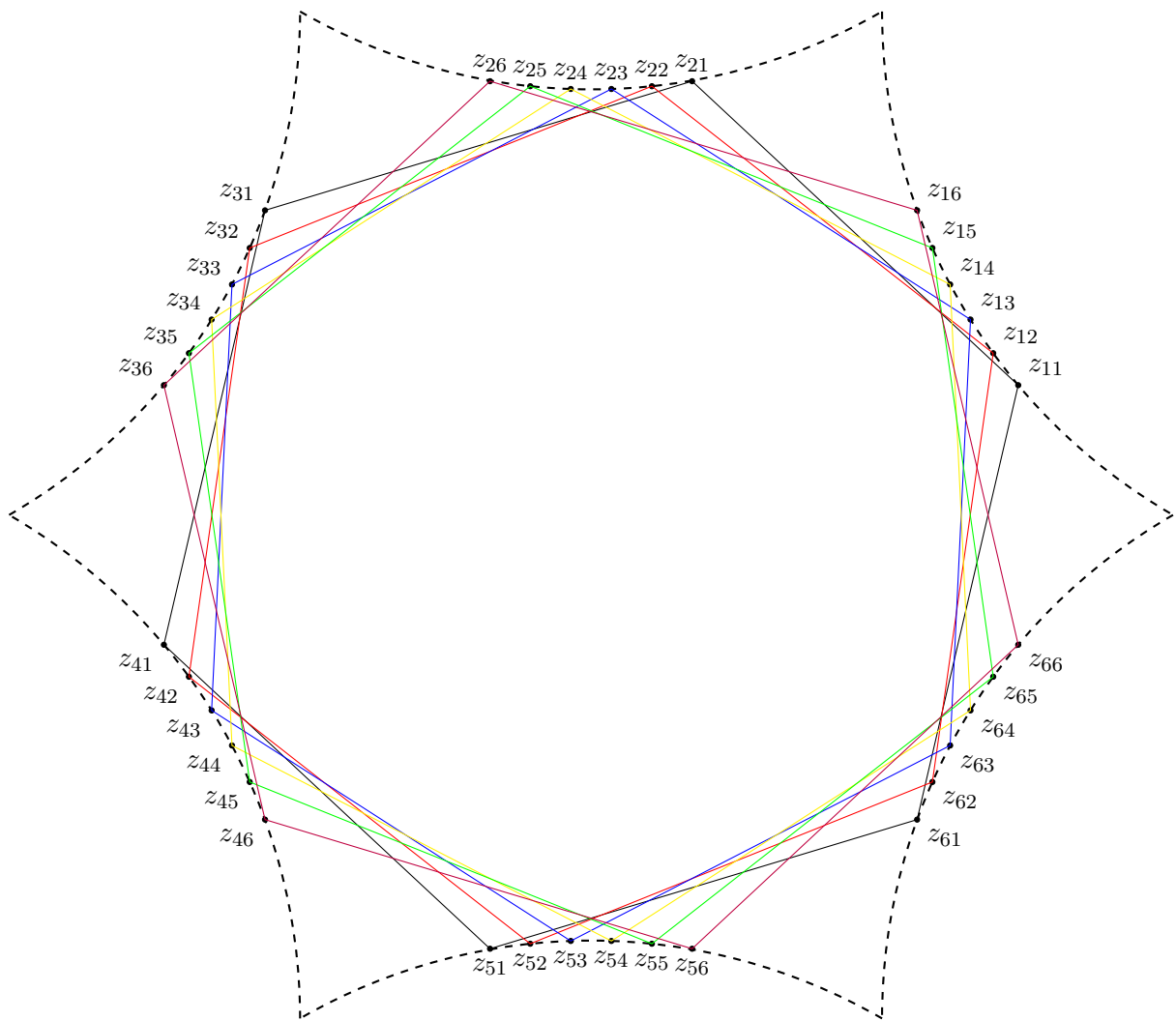
\begin{figure}[ht]
    \centering
\begin{tikzpicture}
 \draw [thick,dashed] (4,6.9282) arc(180:240:8);
 \draw [thick,dashed] (-4,6.9282) arc(240:300:8);
 \draw [thick,dashed] (-8,0) arc(300:360:8);
 \draw [thick,dashed] (-4,-6.9282) arc(0:60:8);
 \draw [thick,dashed] (4,-6.9282) arc(60:120:8);
 \draw [thick,dashed] (8,0) arc(120:180:8);

 \coordinate (z_{11}) at (5.872,1.786);
 \filldraw[black] (z_{11}) circle (1pt) node [above right]
{$z_{11}$};
 \coordinate (z_{12}) at (5.528,2.226);
 \filldraw[black] (z_{12}) circle (1pt) node [above right]
{$z_{12}$};
 \coordinate (z_{13}) at (5.216,2.689);
 \filldraw[black] (z_{13}) circle (1pt) node [above right]
{$z_{13}$};
 \coordinate (z_{14}) at (4.936,3.172);
 \filldraw[black] (z_{14}) circle (1pt) node [above right]
{$z_{14}$};
\coordinate (z_{15}) at (4.692,3.674);
 \filldraw[black] (z_{15}) circle (1pt) node [above right]
{$z_{15}$};
 \coordinate (z_{16}) at (4.482,4.192);
 \filldraw[black] (z_{16}) circle (1pt) node [above right]
{$z_{16}$};
 \coordinate (z_{21}) at (1.389,5.969);
 \filldraw[black] (z_{21}) circle (1pt) node [above]
{$z_{21}$};
\coordinate (z_{22}) at (0.836,5.900);
 \filldraw[black] (z_{22}) circle (1pt) node [above]
{$z_{22}$};
 \coordinate (z_{23}) at (0.279,5.862);
 \filldraw[black] (z_{23}) circle (1pt) node [above]
{$z_{23}$};
 \coordinate (z_{24}) at (-0.279,5.862);
 \filldraw[black] (z_{24}) circle (1pt) node [above]
{$z_{24}$};
 \coordinate (z_{25}) at (-0.836,5.900);
 \filldraw[black] (z_{25}) circle (1pt) node [above]
{$z_{25}$};
 \coordinate (z_{26}) at (-1.389,5.969);
 \filldraw[black] (z_{26}) circle (1pt) node [above]
{$z_{26}$};
 \coordinate (z_{31}) at (-4.482,4.192);
 \filldraw[black] (z_{31}) circle (1pt) node [above left]
{$z_{31}$};
\coordinate (z_{32}) at (-4.692,3.674);
 \filldraw[black] (z_{32}) circle (1pt) node [above left]
{$z_{32}$};
 \coordinate (z_{33}) at (-4.936,3.172);
 \filldraw[black] (z_{33}) circle (1pt) node [above left]
{$z_{33}$};
 \coordinate (z_{34}) at (-5.216,2.689);
 \filldraw[black] (z_{34}) circle (1pt) node [above left]
{$z_{34}$};
\coordinate (z_{35}) at (-5.528,2.226);
 \filldraw[black] (z_{35}) circle (1pt) node [above left]
{$z_{35}$};
 \coordinate (z_{36}) at (-5.872,1.786);
 \filldraw[black] (z_{36}) circle (1pt) node [above left]
{$z_{36}$};
 \coordinate (z_{41}) at  (-5.872,-1.786);
 \filldraw[black] (z_{41}) circle (1pt) node [below left]
{$z_{41}$};
 \coordinate (z_{42}) at (-5.528,-2.226);
 \filldraw[black] (z_{42}) circle (1pt) node [below left]
{$z_{42}$};
 \coordinate (z_{43}) at (-5.216,-2.689);
 \filldraw[black] (z_{43}) circle (1pt) node [below left]
{$z_{43}$};
 \coordinate (z_{44}) at (-4.936,-3.172);
 \filldraw[black] (z_{44}) circle (1pt) node [below left]
{$z_{44}$};
\coordinate (z_{45}) at (-4.692,-3.674);
 \filldraw[black] (z_{45}) circle (1pt) node [below left]
{$z_{45}$};
 \coordinate (z_{46}) at (-4.482,-4.192);
 \filldraw[black] (z_{46}) circle (1pt) node [below left]
{$z_{46}$};
 \coordinate (z_{51}) at (-1.389,-5.969);
 \filldraw[black] (z_{51}) circle (1pt) node [below]
{$z_{51}$};
\coordinate (z_{52}) at (-0.836,-5.900);
 \filldraw[black] (z_{52}) circle (1pt) node [below]
{$z_{52}$};
 \coordinate (z_{53}) at (-0.279,-5.862);
 \filldraw[black] (z_{53}) circle (1pt) node [below]
{$z_{53}$};
 \coordinate (z_{54}) at (0.279,-5.862);
 \filldraw[black] (z_{54}) circle (1pt) node [below]
{$z_{54}$};
 \coordinate (z_{55}) at (0.836,-5.900);
 \filldraw[black] (z_{55}) circle (1pt) node [below]
{$z_{55}$};
 \coordinate (z_{56}) at (1.389,-5.969);
 \filldraw[black] (z_{56}) circle (1pt) node [below]
{$z_{56}$};
 \coordinate (z_{61}) at (4.482,-4.192);
 \filldraw[black] (z_{61}) circle (1pt) node [below right]
{$z_{61}$};
\coordinate (z_{62}) at (4.692,-3.674);
 \filldraw[black] (z_{62}) circle (1pt) node [below right]
{$z_{62}$};
 \coordinate (z_{63}) at (4.936,-3.172);
 \filldraw[black] (z_{63}) circle (1pt) node [below right]
{$z_{63}$};
 \coordinate (z_{64}) at (5.216,-2.689);
 \filldraw[black] (z_{64}) circle (1pt) node [below right]
{$z_{64}$};
\coordinate (z_{65}) at (5.528,-2.226);
 \filldraw[black] (z_{65}) circle (1pt) node [below right]
{$z_{65}$};
 \coordinate (z_{66}) at (5.872,-1.786);
 \filldraw[black] (z_{66}) circle (1pt) node [below right]
{$z_{66}$};

\draw [black](z_{11})--(z_{21})--(z_{31})--(z_{41})--(z_{51})--(z_{61})--(z_{11});
\draw [red](z_{12})--(z_{22})--(z_{32})--(z_{42})--(z_{52})--(z_{62})--(z_{12});
\draw [blue](z_{13})--(z_{23})--(z_{33})--(z_{43})--(z_{53})--(z_{63})--(z_{13});
\draw [yellow](z_{14})--(z_{24})--(z_{34})--(z_{44})--(z_{54})--(z_{64})--(z_{14});
\draw [green](z_{15})--(z_{25})--(z_{35})--(z_{45})--(z_{55})--(z_{65})--(z_{15});
\draw [purple](z_{16})--(z_{26})--(z_{36})--(z_{46})--(z_{56})--(z_{66})--(z_{16});
    
\end{tikzpicture}
\caption{We select \(s = 6\) points on each circular arc. This figure illustrates that each convex polygon intersects each disk only at the selected points.}\label{figure:66666666}
\end{figure}

\subsection{Selection of various geometric parameters}\label{subsection:Selection}
Recall that \(f_{2}(2,1,1)=4\). We first describe the construction for \(f_{2}(2,s,s)> s^{2}\) for \(s\ge 2\). We construct a set \(P\) of points with \(|P|=s^{2}\) as follows.

Fix an integer \(s\ge 2\) and for each \(k\in [s]\), we set
\[ \phi_k:=\frac{(2k+1)\pi}{s}.
\]
Define $M$ to be a sufficiently large real number such that
\begin{equation}\label{eq:choiceofM}
   2(M-1)\sin{\bigg(\frac{\pi}{s}\bigg)\cos{\bigg(\frac{(s+2)\pi}{2s}}\bigg)}+6<0, 
\end{equation}
then we can see
$$M>\frac{3}{(\sin{\frac{\pi}{s}})^{2}}+1.$$
Define a cyclic family of points on the complex plane (identified with \(\mathbb{R}^2\)):
\begin{equation}\label{eq:y_k}
  \bm{y}_k:=M\cdot e^{\,i\phi_k}\quad (k\in[s]).
\end{equation}
For each \(k\in [s]\), let $S_k$ be the closed disk centered at $\bm{y}_k$ with radius
\begin{equation}\label{eq:choiceofR}
    R:=M-1,
\end{equation}
and denote the unit center direction by
\[
\bm{u}_k:=\frac{\bm{y}_k}{\|\bm{y}_k\|}=(\cos\phi_k,\sin\phi_k).
\]
Clearly $\bm{u}_k\in \partial S_k$.

For each \(k\), let $I_k\subseteq\partial S_k$ be the minor arc centered at the point \(\bm{u}_{k}\), with central width
\begin{equation}\label{eq:delta-simple}
\delta:= M^{-2}.
\end{equation}
Equivalently, if $t$ is the central angle (measured at $\bm{y}_k$) from that closest point, then
\[
I_k=\{\bm{y}_k-R\cos{t}\cdot\bm{u}_k - R\sin t\cdot\bm{u}_k^\perp:\ |t|\le \delta/2\},
\]
where $\bm{u}_k^\perp=(-\sin\phi_k,\cos\phi_k)$ is a unit vector orthogonal to $\bm{u}_k$.

\subsection{Selection of points}\label{section:SelectionP}
For each \(k\in [s],\) we place \(s\) points
\[
  \bm{z}_{k,1},\dots,\bm{z}_{k,s}\in I_k
\]
in the clockwise direction along $\partial S_k$. Finally set
\begin{equation}\label{eq:PPP}
      P:=\{\bm{z}_{i,j}: i,j\in[s]\}.
\end{equation}
Note that \(|P|=s^2\). Consider an arbitrary bipartition \(P=P_1\sqcup P_2\).
For each row index \(i\in[s]\) define the row-convex container
\[
  C_i:=\operatorname{Conv}\{\bm{z}_{i,k}\in P_1: k\in[s]\},
\] 
and for each column index \(j\in[s]\) define the column-convex container
\[
  D_j:=\operatorname{Conv}\{\bm{z}_{k,j}\in P_2: k\in[s]\}.
\]
Finally put
\[
  C:=\bigcup_{i=1}^s C_i,\qquad \text{~and~} \qquad D:=\bigcup_{j=1}^s D_j.
\]
By construction, \(C\) and \(D\) are \(s\)-convex sets.

\subsection{Disjointness}
By definitions, we can immediately obtain the following claim.
\begin{claim}\label{Claim:PlanarContain}
    \(P_{1}\subseteq C\) and \(P_{2}\subseteq D\).
\end{claim}
\begin{poc}
    Fix \(\bm{z}_{i,j}\in P_1\). Then \(\bm{z}_{i,j}\) is one of the generators of \(C_i\); hence \(\bm{z}_{i,j}\in C_i\subseteq C\). Similarly, if \(\bm{z}_{i,j}\in P_2\), then it is a generator of \(D_j\) and lies in \(D_j\subseteq D\).
\end{poc}

Then it suffices to show the following lemma.
\begin{lemma}\label{lemma:KeyInConstruction}
   Let \(C\) and \(D\) be two \(s\)-convex sets defined as above, then \(C\cap D=\emptyset.\)
\end{lemma}
\begin{proof}[Proof of Lemma~\ref{lemma:KeyInConstruction}]

It suffices to show the following two propositions.
\begin{prop}\label{claim:Position of C}
    \(C\subseteq\bigcup_{i=1}^{s}S_{i}\) and \(C\cap (\bigcup_{i=1}^{s}\partial S_{i})\subseteq P_{1}. \)
\end{prop}
\begin{proof}[Proof of Proposition~\ref{claim:Position of C}]
   For each $i$, all generators of $C_i$ lie on the circular minor arc $I_i\subset \partial S_i$, hence
\[
C_i \subseteq \operatorname{Conv}(P_1\cap I_i) \subseteq \operatorname{Conv}(I_i).
\]
The convex hull $\operatorname{Conv}(I_i)$ is precisely the circular segment cut off by the chord joining the endpoints of $I_i$, and is contained in the closed disk $S_i$. Thus $C_i\subseteq S_i$ for all $i$, and consequently $C\subseteq \bigcup_{i=1}^s S_i$. Second, we show that for each \(i\in [s]\), \(C_{i}\cap \partial S_{i}\subseteq C_{i}\cap P_{1}\). Assume not, there exists some \(i\) and some point \(\bm{v}\in C_{i}\cap \partial S_{i}\) but \(\bm{v}\notin C_{i}\cap P_{1}\), then for this \(i\), \(\bm{v}\) together with all points of \(C_{i}\cap P_{1}\) are in convex position. Hence, \(\bm{v}\) does not belong to the convex hull of all points in \(C_{i}\cap P_{1}\). However, the later convex hull is \(C_{i}\) itself, which is a contradiction.
\end{proof}

\begin{prop}\label{claim:DDDDDD}
    \(D\subseteq\big({\bigcup_{i=1}^{s}{S_{i}^{\circ}}}\big)^{c}\) and \(D\cap (\bigcup_{i=1}^{s}\partial S_{i})= P_{2}. \)
\end{prop}
\begin{proof}[Proof of Proposition~\ref{claim:DDDDDD}]
       Our first goal is to show the following claim based on our selection of circles.
       \begin{claim}\label{claim:Negative}
           Let \(\ell,k\in [s]\) be distinct integers. For any points \(\bm{v}_{k}\in I_k\) and \(\bm{v}_{\ell}\in I_\ell\), we have \((\bm{v}_{k}-\bm{y}_{k})\cdot (\bm{v}_{k}-\bm{v}_{\ell})<0\).
       \end{claim}

\begin{poc}
We will use the following estimation.
For any $i\in [s]$ and any $\bm{v}_i\in I_i$,
\begin{equation}\label{equation:aa}
    \|\bm{v}_i-\bm{u}_i\|\leq \delta R<\frac {1}{M}.
\end{equation}
Furthermore, for any $1\leq k\neq \ell\leq s$, setting $m=\min\{|\ell-k|,s-|\ell-k|\}$, by the choices of points in~\cref{subsection:Selection}, we have 
$$\angle (\bm{u}_k-\bm{y}_k,\bm{u}_k-\bm{u}_{\ell})=\frac{\pi}{2}+\frac{m\pi}{s}\in\bigg[\frac{s+2}{2s}\pi,\pi\bigg],$$
and
$$\angle(\bm{u}_k,\bm{u_{\ell}})=\frac{2m\pi}{s}\in\bigg[\frac{2\pi}{s},\pi\bigg].$$
Therefore, we have
\begin{equation}\label{equation:Cos}
    \cos\angle (\bm{u}_k-\bm{y}_k,\bm{u}_k-\bm{u}_{\ell})\leq\cos\frac{(s+2)\pi}{2s}
\end{equation}
and 
\begin{equation}\label{equation:norm1}
    \|\bm{u}_k-\bm{u}_{\ell}\|=2\sin{\frac{\angle(\bm{u}_k,\bm{u_{\ell}})}{2}}\in\bigg[2\sin{\frac{\pi}{s}},2\bigg].
\end{equation}
Thus we have
\begin{align*}
(\bm{v}_{k}-\bm{y}_{k})\cdot(\bm{v}_{k}-\bm{v}_{\ell})
&= \bigl[(\bm{v}_{k}-\bm{u}_{k})+(\bm{u}_{k}-\bm{y}_{k})\bigr]
   \cdot \bigl[(\bm{v}_{k}-\bm{u}_{k})+(\bm{u}_{k}-\bm{u}_{\ell})+(\bm{u}_{\ell}-\bm{v}_{\ell})\bigr] \\
&= (\bm{u}_{k}-\bm{y}_{k})\cdot(\bm{u}_{k}-\bm{u}_{\ell})
 + \bigl[(\bm{v}_{k}-\bm{u}_{k})+(\bm{u}_{k}-\bm{y}_{k})\bigr]
   \cdot \bigl[(\bm{v}_{k}-\bm{u}_{k})+(\bm{u}_{\ell}-\bm{v}_{\ell})\bigr] \\
&\quad + (\bm{u}_{k}-\bm{u}_{\ell})\cdot(\bm{v}_{k}-\bm{u}_{k}) \\
&\le (\bm{u}_{k}-\bm{y}_{k})\cdot(\bm{u}_{k}-\bm{u}_{\ell})
 + \bigl(\|\bm{v}_{k}-\bm{u}_{k}\|+\|\bm{u}_{k}-\bm{y}_{k}\|\bigr)
   \bigl(\|\bm{v}_{k}-\bm{u}_{k}\|+\|\bm{u}_{\ell}-\bm{v}_{\ell}\|\bigr) \\
&\quad + \|\bm{u}_{k}-\bm{u}_{\ell}\|\,\|\bm{v}_{k}-\bm{u}_{k}\| \\
&\le R\,\|\bm{u}_{k}-\bm{u}_{\ell}\|\,
   \cos\angle\!\bigl(\bm{u}_{k}-\bm{y}_{k},\,\bm{u}_{k}-\bm{u}_{\ell}\bigr)
 + \Bigl(M+\frac{1}{M}\Bigr)\Bigl(\frac{1}{M}+\frac{1}{M}\Bigr)
 + \frac{2}{M} \\
&\le 2(M-1)\sin{\bigg(\frac{\pi}{s}\bigg)}\cdot \cos\bigg(\frac{(s+2)\pi}{2s}\bigg) + 6 \\
&< 0,
\end{align*}
where the first inequality follows since \(\bm{a}\cdot\bm{b}\le\|\bm{a}\|\cdot\|\bm{b}\|\), the second one comes from \eqref{equation:aa} and~\eqref{equation:norm1}, the third one is due to~\eqref{eq:choiceofR},~\eqref{equation:Cos},~\eqref{equation:norm1} and \(\cos\big(\frac{(s+2)\pi}{2s}\big)<0\), and the final one follows from~\eqref{eq:choiceofM}.
\end{poc}

 


       \begin{claim}\label{claim:Disjointttt}
           For each \(j,k\in [s]\), and for any point \(\bm{v}\in D_{j}\), we have \((\bm{z}_{k,j}-\bm{y}_{k})\cdot (\bm{z}_{k,j}-\bm{v})\le 0\). Moreover, the equality holds if and only if \(\bm{v}\) is exactly \(\bm{z}_{k,j}\).
       \end{claim}
       \begin{poc}
          Recall the definition
\[
D_j=\operatorname{Conv}\{ \boldsymbol{z}_{i,j}\in P_2:\ i\in[s] \}.
\]
Hence for each \(\bm{v}\in D_{j}\), there exist indices $I\subseteq[s]$ and coefficients $\alpha_i\ge 0$ with $\sum_{i\in I}\alpha_i=1$ such that
\[
\boldsymbol{v}=\sum_{i\in I}\alpha_i\,\boldsymbol{z}_{i,j}.
\]
Using linearity of the dot product,
\[
(\boldsymbol{z}_{k,j}-\boldsymbol{y}_k)\cdot(\boldsymbol{z}_{k,j}-\boldsymbol{v})
=\sum_{i\in I}\alpha_i\,(\boldsymbol{z}_{k,j}-\boldsymbol{y}_k)\cdot(\boldsymbol{z}_{k,j}-\boldsymbol{z}_{i,j}).
\]
For $i=k$ the summand is $0$. For $i\neq k$ we invoke Claim~\ref{claim:Negative} with
$\boldsymbol{v}_k=\boldsymbol{z}_{k,j}\in I_k$ and $\boldsymbol{v}_\ell=\boldsymbol{z}_{i,j}\in I_i$ to get
\[
(\boldsymbol{z}_{k,j}-\boldsymbol{y}_k)\cdot(\boldsymbol{z}_{k,j}-\boldsymbol{z}_{i,j})<0.
\]
Therefore, every term in the sum is non-positive and any term with $i\neq k$ is strictly negative.
It follows that
\[
(\boldsymbol{z}_{k,j}-\boldsymbol{y}_k)\cdot(\boldsymbol{z}_{k,j}-\boldsymbol{v})\le 0,
\]
with equality if and only if $\alpha_i=0$ for all $i\neq k$ and $\alpha_k=1$, which also implies that
$\boldsymbol{v}=\boldsymbol{z}_{k,j}$. This finishes the proof.
       \end{poc}
    Fix $j,k\in[s]$. For the disk $S_k=\{\boldsymbol{w}:\|\boldsymbol{w}-\boldsymbol{y}_k\|\le R\}$, the tangent line at $\boldsymbol{z}_{k,j}\in\partial S_k$ has outward normal $\boldsymbol{z}_{k,j}-\boldsymbol{y}_k$, and $S_k$ is contained in the supporting halfspace
\[
\{ \boldsymbol{w}: (\boldsymbol{z}_{k,j}-\boldsymbol{y}_k)\cdot(\boldsymbol{w}-\boldsymbol{z}_{k,j}) \le 0 \}.
\]
By Claim~\ref{claim:Disjointttt}, for any $\boldsymbol{v}\in D_j$ we have
\[
(\boldsymbol{z}_{k,j}-\boldsymbol{y}_k)\cdot(\boldsymbol{z}_{k,j}-\boldsymbol{v})\le 0
\quad\Longleftrightarrow\quad
(\boldsymbol{z}_{k,j}-\boldsymbol{y}_k)\cdot(\boldsymbol{v}-\boldsymbol{z}_{k,j})\ge 0.
\]
Hence $\boldsymbol{v}$ lies in the closed halfspace opposite to the one containing $S_k$, so $\boldsymbol{v}\notin S_k^\circ$. Since it holds for every $k\in [s]$ and every $j\in [s]$, this yields
\[
D_j\subseteq\Bigl(\,\bigcup_{k=1}^s S_k^\circ\Bigr)^{\!c}
\quad\text{and therefore}\quad
D=\bigcup_{j=1}^s D_j\subseteq\Bigl(\,\bigcup_{k=1}^s S_k^\circ\Bigr)^{\!c}.
\]

Moreover, the equality characterization in Claim~\ref{claim:Disjointttt} shows that, for any $\boldsymbol{v}\in D_j$ and any $k\in [s]$,
\[
(\boldsymbol{z}_{k,j}-\boldsymbol{y}_k)\cdot(\boldsymbol{z}_{k,j}-\boldsymbol{v})=0
\ \Longleftrightarrow\ 
\boldsymbol{v}=\boldsymbol{z}_{k,j}.
\]
Geometrically, the only point of $D_j$ lying on the tangent line at $\boldsymbol{z}_{k,j}$ is the tangent point itself. Hence
\[
D_j\cap\partial S_k=\{\boldsymbol{z}_{k,j}\}\cap D_j,
\]
and since the generators of $D_j$ are precisely the column-$j$ points that belong to $P_2$, we obtain
\[
D\cap\Bigl(\,\bigcup_{k=1}^s\partial S_k\Bigr)
=\bigcup_{j=1}^s\ \bigcup_{k=1}^s\bigl(D_j\cap\partial S_k\bigr)
=\bigcup_{j=1}^s\ \bigcup_{k=1}^s \bigl(\{\boldsymbol{z}_{k,j}\}\cap P_2\bigr)
= P_2.
\]
This finishes the proof of Proposition~\ref{claim:DDDDDD}.
\end{proof}
Propositions~\ref{claim:Position of C} and~\ref{claim:DDDDDD} together yield \(C\cap D=\emptyset,\) finishing the proof of Lemma~\ref{lemma:KeyInConstruction}.
\end{proof}

This finishes the proof of \(f_{2}(2,s,s)>s^{2}\).

\section{High-dimensional torus and \(f_{r}(d,s,\ldots,s)>s^{r}\) for \(d\ge 2r-2\)}\label{section:Generalization High}
We derive a higher-dimensional construction from the planar one in this section and provide the proof of~\cref{thm:highdLB}. By monotonicity of \(f_{r}(d,s,\ldots,s)\) in the dimension parameter, it suffices to show \(f_{r}(2r-2,s,\ldots,s)>s^{r}\). Notice that Tverberg's theorem is equivalent to \(f_{r}(2r-2,1,1,\ldots,1)=(r-1)(2r-1)+1\), we then set $r\ge 2$ and $s\ge 2$.

\paragraph{Selection of \(s^{r}\) many points.}
Let $\{\bm{z}_{i_1,i_2}:i_1,i_2\in[s]\}\subseteq\mathbb{R}^2$ be the planar point set from the proof of
$f_2(2,s,s)> s^2$ (constructed on circular caps, see~\eqref{eq:PPP} in~\cref{section:SelectionP}).
Let 
\[
U=\{\bm{u}_k:=e^{2\pi i k/s}:k\in[s]\}\subseteq\mathbb{R}^2
\]
be the set of vertices of a regular $s$-gon (viewed in $\mathbb{R}^2\cong\mathbb{C}$).
For $(i_1,\dots,i_r)\in[s]^r$ define a point in $\mathbb{R}^{2r-2}$ by
\[
\bm{p}_{(i_1,\ldots,i_r)}:=\bigl(\bm{z}_{i_1,i_2},\bm{u}_{i_3},\bm{u}_{i_4},\ldots,\bm{u}_{i_r}\bigr).
\]
Set
\[
P:=\{\bm{p}_{(i_1,\ldots,i_r)}:(i_1,\ldots,i_r)\in[s]^r\}\subseteq\mathbb{R}^{2r-2}\cong {\mathbb{R}^2}
\times\underbrace{\mathbb{R}^2\times\cdots\times\mathbb{R}^2}_{r-2\ \text{blocks}}:=\mathbb{G}_{0}\times \mathbb{G}_{1}\times\cdots\times\mathbb{G}_{r-2},
\]
where we use \(\mathbb{G}_{i}\) to denote the \(i\)-th block.
Then it is easy to see that \(|P|=s^r\).

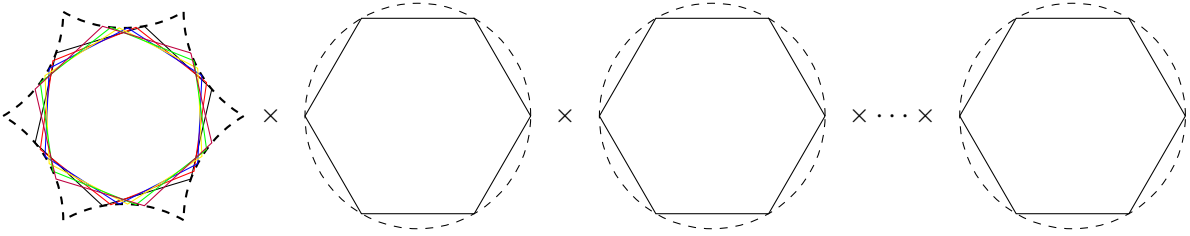
\begin{figure}[ht]
\begin{center}
\begin{minipage}{0.2\textwidth}
\centering
\begin{tikzpicture}[scale=0.2]
 \draw [thick,dashed] (4,6.9282) arc(180:240:8);
 \draw [thick,dashed] (-4,6.9282) arc(240:300:8);
 \draw [thick,dashed] (-8,0) arc(300:360:8);
 \draw [thick,dashed] (-4,-6.9282) arc(0:60:8);
 \draw [thick,dashed] (4,-6.9282) arc(60:120:8);
 \draw [thick,dashed] (8,0) arc(120:180:8);

 \coordinate (z_{11}) at (5.872,1.786);
 \coordinate (z_{12}) at (5.528,2.226);
 \coordinate (z_{13}) at (5.216,2.689);
 \coordinate (z_{14}) at (4.936,3.172);
\coordinate (z_{15}) at (4.692,3.674);
 \coordinate (z_{16}) at (4.482,4.192);
 \coordinate (z_{21}) at (1.389,5.969);
\coordinate (z_{22}) at (0.836,5.900);
 \coordinate (z_{23}) at (0.279,5.862);
 \coordinate (z_{24}) at (-0.279,5.862);
 \coordinate (z_{25}) at (-0.836,5.900);
 \coordinate (z_{26}) at (-1.389,5.969);
 \coordinate (z_{31}) at (-4.482,4.192);
\coordinate (z_{32}) at (-4.692,3.674);
 \coordinate (z_{33}) at (-4.936,3.172);
 \coordinate (z_{34}) at (-5.216,2.689);
\coordinate (z_{35}) at (-5.528,2.226);
 \coordinate (z_{36}) at (-5.872,1.786);
 \coordinate (z_{41}) at  (-5.872,-1.786);
 \coordinate (z_{42}) at (-5.528,-2.226);
 \coordinate (z_{43}) at (-5.216,-2.689);
 \coordinate (z_{44}) at (-4.936,-3.172);
\coordinate (z_{45}) at (-4.692,-3.674);
 \coordinate (z_{46}) at (-4.482,-4.192);
 \coordinate (z_{51}) at (-1.389,-5.969);
\coordinate (z_{52}) at (-0.836,-5.900);
 \coordinate (z_{53}) at (-0.279,-5.862);
 \coordinate (z_{54}) at (0.279,-5.862);
 \coordinate (z_{55}) at (0.836,-5.900);
 \coordinate (z_{56}) at (1.389,-5.969);
 \coordinate (z_{61}) at (4.482,-4.192);
\coordinate (z_{62}) at (4.692,-3.674);
 \coordinate (z_{63}) at (4.936,-3.172);
 \coordinate (z_{64}) at (5.216,-2.689);
\coordinate (z_{65}) at (5.528,-2.226);
 \coordinate (z_{66}) at (5.872,-1.786);

\draw [black](z_{11})--(z_{21})--(z_{31})--(z_{41})--(z_{51})--(z_{61})--(z_{11});
\draw [red](z_{12})--(z_{22})--(z_{32})--(z_{42})--(z_{52})--(z_{62})--(z_{12});
\draw [blue](z_{13})--(z_{23})--(z_{33})--(z_{43})--(z_{53})--(z_{63})--(z_{13});
\draw [yellow](z_{14})--(z_{24})--(z_{34})--(z_{44})--(z_{54})--(z_{64})--(z_{14});
\draw [green](z_{15})--(z_{25})--(z_{35})--(z_{45})--(z_{55})--(z_{65})--(z_{15});
\draw [purple](z_{16})--(z_{26})--(z_{36})--(z_{46})--(z_{56})--(z_{66})--(z_{16});
    
\end{tikzpicture}
\end{minipage}
$\times$
\begin{minipage}{0.2\textwidth}
\centering
\begin{tikzpicture}[scale=0.5]
\draw [dashed] (0,0) circle (3);
\coordinate (z_1) at (3,0);
\coordinate (z_2) at (1.5,2.6);
\coordinate (z_3) at (-1.5,2.6);
\coordinate (z_4) at (-3,0);
\coordinate (z_5) at (-1.5,-2.6);
\coordinate (z_6) at (1.5,-2.6);
\draw (z_1)--(z_2)--(z_3)--(z_4)--(z_5)--(z_6)--(z_1);
\end{tikzpicture}
\end{minipage}
$\times$
\begin{minipage}{0.2\textwidth}
\centering
\begin{tikzpicture}[scale=0.5]
\draw [dashed] (0,0) circle (3);
\coordinate (z_1) at (3,0);
\coordinate (z_2) at (1.5,2.6);
\coordinate (z_3) at (-1.5,2.6);
\coordinate (z_4) at (-3,0);
\coordinate (z_5) at (-1.5,-2.6);
\coordinate (z_6) at (1.5,-2.6);
\draw (z_1)--(z_2)--(z_3)--(z_4)--(z_5)--(z_6)--(z_1);
\end{tikzpicture}
\end{minipage}
$\times \cdots \times$
\begin{minipage}{0.2\textwidth}
\centering
\begin{tikzpicture}[scale=0.5]
\draw [dashed] (0,0) circle (3);
\coordinate (z_1) at (3,0);
\coordinate (z_2) at (1.5,2.6);
\coordinate (z_3) at (-1.5,2.6);
\coordinate (z_4) at (-3,0);
\coordinate (z_5) at (-1.5,-2.6);
\coordinate (z_6) at (1.5,-2.6);
\draw (z_1)--(z_2)--(z_3)--(z_4)--(z_5)--(z_6)--(z_1);
\end{tikzpicture}
\end{minipage}
\end{center}
\caption{An illustration of the high-dimensional construction.}
\end{figure}

\paragraph{Adversarial partition and $s$-convex containers.}
Let $P=P_1\sqcup\cdots\sqcup P_r$ be an arbitrary partition.
For each $j\in[r]$ and $k\in[s]$, define the $j$-th \emph{layer} of index $k$ by
\[
G_{j,k}:=\{\bm{p}_{(i_1,\ldots,i_r)}\in P_j: i_j=k\,\}.
\]
Define $C_{j,k}:=\operatorname{Conv}(G_{j,k})$ and $C_j:=\bigcup_{k=1}^s C_{j,k}$.
By definition, $P_j\subseteq C_j$ and each $C_j$ is an $s$-convex set.

We then prove the following claim, which immediately implies $f_r(2r-2,s,\ldots,s)> s^r$.

\begin{claim}
    For all $(k_1,\ldots,k_r)\in[s]^{r}$, we have
\[
C_{1,k_1}\ \cap\ C_{2,k_2}\ \cap\ \cdots\ \cap\ C_{r,k_r}\ =\ \emptyset.
\]
\end{claim}
\begin{poc}
    Recall the ambient identification
\[
\mathbb{G}:=\mathbb{R}^{2r-2}\ \cong\  {\mathbb{R}^2}
\times\underbrace{\mathbb{R}^2\times\cdots\times\mathbb{R}^2}_{r-2\ \text{blocks}}:=\mathbb{G}_{0}\times \mathbb{G}_{1}\times\cdots\times\mathbb{G}_{r-2},
\]
where we use \(\mathbb{G}_{i}\) to denote the \(i\)-th block.
Let
\[
\pi^{(1)}: \mathbb{G}\rightarrow \mathbb{G}_{0}\ \textup{such\ that}\ \pi^{(1)}(\bm{z}_{i_{1},i_{2}},\bm{u}_{i_{3}},\ldots,\bm{u}_{i_{r}})=\bm{z}_{i_{1},i_{2}}.
\]
Moreover, let
\[
\pi^{(>1)}:\mathbb{G}\rightarrow\mathbb{G}_{1}\times\cdots\times\mathbb{G}_{r-2}\ \textup{such\ that}\ \pi^{(>1)}(\bm{z}_{i_{1},i_{2}},\bm{u}_{i_{3}},\ldots,\bm{u}_{i_{r}})=(\bm{u}_{i_{3}},\ldots,\bm{u}_{i_{r}}).
\]

Suppose that there exists some $(k_1,\dots,k_r)\in[s]^{r}$ such that
\[
\boldsymbol{v}\ \in\ C_{1,k_1}\cap C_{2,k_2}\cap\cdots\cap C_{r,k_r}.
\]
Write $\boldsymbol{v}=(\boldsymbol{v}^{(1)},\boldsymbol{v}^{(>1)})$ with
$\boldsymbol{v}^{(1)}=\pi^{(1)}(\boldsymbol{v})$ and
$\boldsymbol{v}^{(>1)}=\pi^{(>1)}(\boldsymbol{v})$.

For $j\ge 3$, by the definition of \(G_{j,k}\), we can see 
\[
G_{j,k}\subseteq \mathbb{G}_{0}\times\mathbb{G}_{1}\times\cdots\times\mathbb{G}_{j-3}\times\{\bm{u}_{k}\}\times\mathbb{G}_{j-1}\times\cdots\times\mathbb{G}_{r-2}.
\]
Since $\boldsymbol{v}\in C_{j,k_j}=\operatorname{Conv}(G_{j,k_j})$, the convexity implies that 
\[
\bm{v}\in\bigcap_{j=3}^{r}(\mathbb{G}_{0}\times\mathbb{G}_{1}\times\cdots\times\mathbb{G}_{j-3}\times\{\bm{u}_{k_{j}}\}\times\mathbb{G}_{j-1}\times\cdots\times\mathbb{G}_{r-2}).
\]
Therefore, we have
\begin{equation}\label{qe:v>1}
    \boldsymbol{v}^{(>1)}=(\boldsymbol{u}_{k_3},\boldsymbol{u}_{k_4},\dots,\boldsymbol{u}_{k_r}).
\end{equation}

Recall that $U:=\{\boldsymbol{u}_1,\ldots,\boldsymbol{u}_s\}\subseteq\mathbb{R}^2$ forms a regular \(s\)-gon. By Fact~\ref{fact:COnvexPosition}(2)-(3), $U\times \cdots \times U$ is in convex position and
\[
\textup{Ext}(U\times\cdots\times U)=U\times\cdots\times U.
\] Since $\boldsymbol{v}^{(>1)}\in U\times\cdots\times U$, it is an extreme point of $U\times\cdots\times U$.

Furthermore, since $\boldsymbol{v}\in C_{1,k_1}=\operatorname{Conv}(G_{1,k_1})$, we can select a minimal finite set
$X\subseteq G_{1,k_1}$ with $\boldsymbol{v}\in\operatorname{Conv}(X)$. The following claim is the key step.

\begin{claim}\label{claim:Points In X}
 For every \(\bm{x}\in X\), \(\pi^{(>1)}(\boldsymbol{x})=\boldsymbol{v}^{(>1)}\).
\end{claim}
\begin{poc}
Let \(X:=\{\bm{x}_{1},\bm{x}_{2},\ldots,\bm{x}_{m}\}\). Since \(\bm{v}\in\textup{Conv}(X)\), we can write \(\bm{v}\) as 
\begin{equation}\label{repv}
\bm{v}=\sum\limits_{i=1}^{m}\alpha_{i}\bm{x}_{i},  
\end{equation}
where \(\alpha_{i}>0\) due to the minimality of \(X\). By the definition of \(\pi^{(>1)}\) and~\eqref{repv} we have
\[
\bm{v}^{(>1)}=\sum\limits_{i=1}^{m}\alpha_{i}\cdot\pi^{(>1)}(\bm{x}_{i}).
\]
Suppose the claim is false, without loss of generality and by a suitable relabeling, there exists some \(q\in [m]\) such that \(\pi^{(>1)}(\bm{x}_{i})\neq \bm{v}^{(>1)}\) for any \(i\le q\) and \(\pi^{(>1)}(\bm{x}_{i})= \bm{v}^{(>1)}\) otherwise. By the fact that $\sum\limits_{i=1}^m \alpha_i=1$, we have
\[
\sum_{i\le q}\alpha_{i}\bm{v}^{(>1)}=\sum\limits_{i\le q}\alpha_{i}\pi^{(>1)}(\bm{x}_{i}),
\]
which implies that
\[
\bm{v}^{(>1)}=\sum\limits_{i\le q}\bigg(\frac{\alpha_{i}}{\sum_{i\le q}\alpha_{i}}\bigg)\pi^{(>1)}(\bm{x}_{i}).
\]
However, this contradicts that \(\bm{v}^{(>1)}\) is an extreme point of $U\times\cdots\times U$. This finishes the proof.
\end{poc}

Then by Claim~\ref{claim:Points In X} and since \(X\subseteq G_{1,k_{1}}\), every $\boldsymbol{x}\in X$ has the form
\[
\boldsymbol{x}=\bigl(\boldsymbol{z}_{k_1,i_2},\ \boldsymbol{v}^{(>1)}\bigr)
\quad\text{for some }i_2\in[s],
\]
where $\boldsymbol{z}_{k_1,i_2}\in\mathbb{G}_{0}$ comes from the planar construction.
By definition, we have
\begin{equation}\label{eq:row-hull}
\boldsymbol{v}^{(1)}\in \operatorname{Conv}\bigl(\{\boldsymbol{z}_{k_1,i_2}:\ (\boldsymbol{z}_{k_1,i_2},\boldsymbol{v}^{(>1)})\in P_1\}\bigr).
\end{equation}
An entirely analogous argument for $\boldsymbol{v}\in C_{2,k_2}$ provides a minimal finite set
$Y\subseteq G_{2,k_2}$ with $\boldsymbol{v}\in\operatorname{Conv}(Y)$ and
$\pi^{(>1)}(\boldsymbol{y})=\boldsymbol{v}^{(>1)}$ for all $\boldsymbol{y}\in Y$, whence
\begin{equation}\label{eq:col-hull}
\boldsymbol{v}^{(1)}\in \operatorname{Conv}\bigl(\{\boldsymbol{z}_{i_1,k_2}:\ (\boldsymbol{z}_{i_1,k_2},\boldsymbol{v}^{(>1)})\in P_2\}\bigr).
\end{equation}

Define the row and column slices at the layer $\boldsymbol{v}^{(>1)}$ respectively by
\[
A:=\{\boldsymbol{z}_{k_1,i_2}:\ (\boldsymbol{z}_{k_1,i_2},\boldsymbol{v}^{(>1)})\in P_1\},\qquad
B:=\{\boldsymbol{z}_{i_1,k_2}:\ (\boldsymbol{z}_{i_1,k_2},\boldsymbol{v}^{(>1)})\in P_2\}.
\]
Then \eqref{eq:row-hull}–\eqref{eq:col-hull} state that
\[
\boldsymbol{v}^{(1)}\in \operatorname{Conv}(A) \cap \operatorname{Conv}(B)\ \subseteq \mathbb{G}_{0}.
\]
However, $A$ lies in the $k_1$-th row and $B$ lies in the $k_2$-th column of the planar point set,
selected according to the induced partition on the fixed higher layer $\boldsymbol{v}^{(>1)}$.
By the planar separation property established in Lemma~\ref{lemma:KeyInConstruction} (the union of row-hulls from one part and the union of column-hulls from the other part are disjoint), \(\textup{Conv}(A)\cap \textup{Conv}(B)=\emptyset\),
which is a contradiction. This finishes the proof.
\end{poc}

\section{The power of disjointness and almost linear bound for \(F_{2}(2,s,s)\)}
In this section, we provide an upper bound for \(F_{2}(2,s,s)\). For convenience, we recall Definition~\ref{def:Fr}.
\begin{defn-non}\label{def:disjointness}
For integers $d\ge1$, $r\ge2$, and $s_1,\ldots,s_r\ge1$, let $F_r(d,s_1,\ldots,s_r)$ be the least integer $n$ such that every $n$-point set $P\subseteq\mathbb{R}^d$ admits a partition
$P=P_1\sqcup\cdots\sqcup P_r$ with the following property: If for each \(i\in [r]\), \(C_{i}\) is a union of \(s_{i}\) convex sets \(C_{i,j}\) (\(j\in [s_i]\)) such that $C_{i,j}\cap C_{i,j'}=\emptyset$ for all $j\ne j'$ (that is, \(C_{i}=\bigsqcup_{j=1}^{s_{i}}C_{i,j}\)) and \(P_{i}\subseteq C_{i}\), then $\bigcap_{i=1}^{r} C_i\neq\emptyset$ must hold.
\end{defn-non}

To show \(F_{r}(d,s,\ldots,s)\le F\), it suffices to show that for any point set \(P\subseteq \mathbb{R}^{d}\) of size at least \(F\), there exists some partition \(P=\bigsqcup_{i=1}^{r}P_{i}\) such that for any family of $s_i$-convex sets $C_{i}=\bigsqcup_{j=1}^{s_{i}}C_{i,j}$ containing \(P_{i}\) with \(i\in [r]\), we have \(\bigcap_{i=1}^{r}C_{i}\neq\emptyset.\)

\subsection{Tools and auxiliary results}\label{sec:TOOLs}
In this part, we will take advantage of various results in the fields of extremal combinatorics and discrete geometry. 
The following lemma was shown in~\cite[Lemma~2.2]{2025arxivAlonSmo}, and we also refer the interested readers to the great book~\cite{2002BookMatousek}.
\begin{lemma}[\cite{2025arxivAlonSmo}]\label{lem:AS-VC}
Fix integers \(d\ge 1\) and \(\ell\ge 1\).
Let \(\mathcal{R}_{d,\ell}\) be the family of subsets of \(\mathbb{R}^d\) of the form
\[
R_{d,\ell} = \bigcup_{i=1}^{m} P_i,
\]
where each \(P_i\) is a convex polyhedron in \(\mathbb{R}^d\) and the \emph{total} number of facets of \(P_1,\dots,P_m\) is at most \(\ell\).
Then the range space \(\bigl(\mathbb{R}^d,\mathcal{R}_{d,\ell}\bigr)\) has VC-dimension at most \(c d \ell \log \ell\) for some absolute constant \(c>0\).
Equivalently, there is a universal constant \(c>0\) such that no point set in \(\mathbb{R}^d\) of size larger than \(c d \ell \log \ell\) can be shattered by \(\mathcal{R}_{d,\ell}\).
\end{lemma}

One of the main contributions in~\cite{2025arxivAlonSmo} is the following extension of VC-dimension and near-optimal theoretical bounds.
\begin{defn}[$r$-shattered set~\cite{2025arxivAlonSmo}]
Let $H=(V,E)$ be a fixed hypergraph. A subset $S\subseteq V$ is said to be \emph{$r$-shattered} by $E$ if for any partition of $S$ into $r$ pairwise disjoint sets $S_i$ (that is, $S=\bigsqcup_{i=1}^{r} S_i$) there exist hyperedges $e_1,\ldots,e_r\in E$ such that $S_i\subseteq e_i$ for all $i\in[r]$ and
\(
S \cap \big(\bigcap_{i=1}^{r} e_i\big) = \emptyset.
\)
\end{defn}

\begin{lemma}[\cite{2025arxivAlonSmo}]\label{lemma:RSS}
There exists an absolute constant $c$ such that for every integer $d$, any hypergraph $H=(V,E)$ with VC-dimension $d$, and every integer $r\ge 2$, every $r$-shattered set by $E$ has size at most $cdr^{2}\log r$. This bound is nearly optimal: for every $d$ and $r$ there is a hypergraph with VC-dimension $d$ that admits an $r$-shattered set of size $\Omega(dr^{2})$.
\end{lemma}

\subsection{Planar line-separation: Proof of~\cref{thm:VariantUBPlanar}}
We first introduce the following definition. 
\begin{defn}[Separating systems]
   Let \(s\ge 1\) be an integer. Fix pairwise disjoint, nonempty compact convex sets \(D_1,\dots,D_s\subseteq\mathbb{R}^2\).
A \emph{separating system} for \(\mathcal{D}=\{D_{1},\ldots,D_{s}\}\) is a family \(P:=\{P_1,\ldots,P_s\}\), where \(P_{i}\) is a convex polyhedron, such that \(D_i\subseteq P_i\) for each \(i\in[s]\), and \(P_{i}^\circ \bigcap P_{j}^\circ =\emptyset \) for any \(i\neq j\).
\end{defn}
Let \(\mathcal{P}_{\mathcal{D}}\) be the family of all separating systems for \(\mathcal{D}=\{D_{1},\ldots,D_{s}\}\). We remark that for any \(\mathcal{D}\), \(\mathcal{P}_{\mathcal{D}}\) cannot be empty by~\cite[Lemma~2.1]{2025arxivAlonSmo}. For \(P=\{P_i\}_{i=1}^s\in\mathcal{P}_{\mathcal{D}}\), define the total number of facets
\[
\mathrm{Fac}(P)\;:=\;\sum_{i=1}^s \#\{\text{facets\ of\ }P_i\}.
\]
We then define the key parameter as follows.

\begin{defn}
   For an integer \(a\ge 1\), let \(g(a)\) be the smallest integer with the following property: For every family \(\mathcal{D}=\{D_{1},\ldots,D_{a}\}\subseteq \mathbb{R}^{2}\) of pairwise disjoint convex sets, there exists a separating system \(P\) for \(\mathcal{D}\) with \(\textup{Fac}(P)\le g(a).\)
\end{defn}

It was shown in~\cite{1990DM} that \(g(a)\) grows linearly in \(a\).

\begin{theorem}[Theorem 2~\cite{1990DM}]\label{thm:gs}
For every integer \(a\ge 3\), \(g(a)\le 6a-9\).
\end{theorem}

We then establish the following relation between \(F_{2}(2,s,s)\) and the above function.

\begin{prop}\label{prop:Relation}
There exists an absolute constant \(c>0\) such that for all integers \(s\ge 1\),
\[
F_{2}(2,s,s) \le cg(2s)\log({g(2s)}).
\]
\end{prop}

\begin{proof}[Proof of Proposition~\ref{prop:Relation}]
The proof consists of two parts.
Let \(\ell=g(2s)\). Let \(\mathcal{H}=(P,E)\) be a hypergraph, where \(P\) is a set of points and \(S\subseteq P\) forms an edge if and only if there exist some positive integer $t$, and interior-disjoint polyhedron \(K_{1},\ldots,K_{t}\) such that these polyhedrons contain at most $\ell$ facets in total, and that \(S=P\cap \bigg(\bigcup\limits_{i=1}^{t}K_{i}\bigg)\), namely \(S\) can be cut from \(P\) by intersecting it with a separating system \(R=\{K_i\}_{i=1}^t\) with \(\mathrm{Fac}(R)\leq \ell\). By Lemma~\ref{lem:AS-VC}, the VC-dimension of \(\mathcal{H}\) is at most \(c_{\eqref{lem:AS-VC}}\ell\log{\ell}\). Equivalently, no point set of size \(c_{\eqref{lem:AS-VC}}\ell\log{\ell}+1\) can be shattered by \(\mathcal{H}\).

Suppose, for contradiction, that \(F_{2}(2,s,s) > 2c_{\eqref{lem:AS-VC}}\ell\log{\ell}\).
By the definition of \(F_{2}(2,s,s)\), there exists a set \(P'\subseteq\mathbb{R}^2\) of size \(|P'|=F_{2}(2,s,s)-1\ge c_{\eqref{lem:AS-VC}}\ell\log{\ell}+1\) such that for every partition \(P'=A'\sqcup B'\) there are two {disjoint} \(s\)-convex sets \(C_1, C_2\subseteq\mathbb{R}^2\) with
\[
C_1=\bigsqcup_{i=1}^{s} X_i\ \text{(each }X_i\text{ convex)},\qquad
C_2=\bigsqcup_{j=1}^{s} Y_j\ \text{(each }Y_j\text{ convex)},
\]
satisfying \(A'\subseteq C_1\), \(B'\subseteq C_2\), and \(C_1\cap C_2=\emptyset\).

Apply the definition of \(g(2s)\) to the family \(\{X_1,\dots,X_s,Y_1,\dots,Y_s\}\) of at most \(2s\) pairwise disjoint convex sets: there exists a separating system \(\mathcal{D}=\{D_{1},\ldots,D_{2s}\}\) such that 
\(X_{i}\subseteq D_{i}\) for each \(i\in [s]\) and \(Y_{j}\subseteq D_{s+j}\) for each \(j\in [s]\).
Let 
\[
\mathcal{D}_{X}:=\bigcup_{i=1}^{s}D_{i}.
\]
Furthermore, by definition of the separating system, \(C_1\subseteq \mathcal{D}_{X}\) and \(\mathcal{D}_{X}\cap C_2=\emptyset\).
Consequently,
\[
\mathcal{D}_{X}\cap P'=A'\qquad\text{and}\qquad (\mathbb{R}^2\setminus \mathcal{D}_{X})\cap P'=B'.
\]
Since the partition \(P'=A'\sqcup B'\) was arbitrary, this shows that \(P'\) is shattered by \(\mathcal{H}\), contradicting the VC-dimension bound for \(\mathcal{H}\).
Therefore \(|P'|\le 2c_{\eqref{lem:AS-VC}}\ell\log{\ell}\), which yields that
\[
F_{2}(2,s,s)\le 2c_{\eqref{lem:AS-VC}}g(2s)\log({g(2s)}).
\]
Let \(c=2c_{\eqref{lem:AS-VC}}\), then
\[
F_{2}(2,s,s)\le cg(2s)\log({g(2s)}).
\]
This finishes the proof.
\end{proof}

 Then combining Theorem~\ref{thm:gs} with the reduction above immediately yields~\cref{thm:VariantUBPlanar}.

\section{Upper bound for \(F_{r}(d,s,\ldots,s)\) via a variant of Erd\H{o}s box problem }
A central theme in extremal combinatorics is to determine the largest possible size of a set system or a hypergraph that avoids a prescribed forbidden configuration. Formally, for a given integer \(d\ge 2\) and a given hypergraph \(H\), we denote by \(\textup{ex}_{d}(n,H)\) the maximum number of edges of an \(n\)-vertex \(d\)-uniform hypergraph that does not contain \(H\). One of the earliest results in hypergraph Tur\'{a}n problem is due to Erd\H{o}s~\cite{1964Israel}, which focuses on {forbidding} complete \(d\)-partite \(d\)-uniform hypergraph \(K_{s_{1},s_{2},\ldots,s_{d}}^{(d)}\). In particular, when \(s_{1}=s_{2}=\cdots=s_{d}=2\), the problem is widely known as the \emph{Erd\H{o}s box problem}~\cite{2021DA,1999JCTA,2002MRL}.

\begin{theorem}[\cite{1964Israel}]\label{erdosBox}
    Let \(d\ge 2\) and \(r_{1}\le\cdots \le r_{d}\) be positive integers and let \(n\) be a sufficiently large integer. Then
    \begin{equation}\label{equation:ErdosBoxUpperbound}
    \textup{ex}_{d}(n,K_{r_{1},r_{2},\ldots,r_{d}}^{(d)})\le c_{r_{1},r_{2},\ldots,r_{d}}\cdot n^{d-\frac{1}{\prod_{i=1}^{d-1}r_{i}}},
    \end{equation}
    where \(c_{r_{1},r_{2},\ldots,r_{d}}>0\) is a constant depending on \(r_{1},r_{2},\ldots,r_{d}\).
\end{theorem}

The separating theorem is one of the most fundamental results in the field of discrete geometry. we state it here for completeness.
\begin{theorem}[\cite{2021CombConvex}]\label{thm:SeparatingTheorem}
Assume \(C\) and \(K\) are convex sets in \(\mathbb{R}^{n}\), where \(C\) is compact and \(K\) is closed. Then \(C\cap K=\emptyset\) if and only if there are closed halfspaces \(\mathbb{H}_{1}\) and \(\mathbb{H}_{2}\) such that \(C\subseteq \mathbb{H}_{1}\), \(K\subseteq \mathbb{H}_{2}\), and \(\mathbb{H}_{1}\cap\mathbb{H}_{2}=\emptyset.\) 
Equivalently, \(C\cap K=\emptyset\) if and only if there exists a hyperplane $H$ such that $C$ and $K$ lie in the two connected components of $\mathbb{R}^{n}\backslash H$ respectively.
\end{theorem} 

The main goal of this section is to prove a slightly better bound than the one stated in~\cref{thm:GeneralUpperbound} via the above result.

\begin{theorem}\label{thm:GUP}
  Let $r$ and $d$ be two positive integers such that $r\ge d+2$, and let $s$ be a sufficiently large integer relative to $r$ and $d$. Then there exists some constant $c_{r,d}>0$ such that
\[
 F_r(d,s,\ldots,s)
  \le 
 c_{r,d}\cdot \min\bigg\{s^{\bigl(1-\frac{1}{2^{d}(d+1)}\bigr)r+\frac{1}{2^d}}, s^{2d+3-2^{-r}}\bigg\}\cdot \log s.
\]
\end{theorem}
\begin{proof}[Proof of~\cref{thm:GUP}]
Recall that to show \(F_{r}(d,s,\ldots,s)\le F\), it suffices to show that for any point set \(P\subseteq\mathbb{R}^{d}\) with \(|P|\ge F\), there exists some partition \(P=\bigsqcup_{i=1}^{r}P_{i}\) such that the following holds: for any families of convex sets 
$\{C_{i,1},\ldots,C_{i,s}\}_{i\in[r]}$ such that $C_{i,j}\cap C_{i,j'}=\emptyset$ for all $j\ne j'$, the unions 
$C_i:=\bigsqcup_{j=1}^{s} C_{i,j}$ satisfy \(P_{i}\subseteq C_{i}\) for each \(i\in [r]\) and $\bigcap_{i=1}^{r} C_i\neq\emptyset$. 

Our strategy is to focus on the following enumerative problem.

\begin{prop}\label{prop:Counting}
Let $r$ and $d$ be positive integers such that $r\ge d+2$ and let $s$ be sufficiently large relative to $r$ and $d$. Let \(C_{1},\ldots,C_{r}\) be \(s\)-convex sets in \(\mathbb{R}^{d}\) with \(\bigcap_{j=1}^{r}C_{j}=\emptyset\), where for each \(i\in [r]\),  $C_{i}=\bigsqcup_{k=1}^{s}C_{i,k}$ with $C_{i,j}\cap C_{i,j'}=\emptyset$ for all $j\ne j'$. Then there exists a large constant \(c_{r,d}'>0\) such that for any \((r-1)\)-subset \(A\subseteq [r]\), the number of \((r-1)\)-tuples \((i_{\ell})_{\ell\in A}\in [s]^{r-1}\) with \(\bigcap_{\ell\in A}C_{\ell,i_{\ell}}\neq\emptyset\) is at most \(q\), where \(q=q(d,s,r):= \min \bigl \{c_{r,d}'\cdot s^{(1-\frac{1}{2^d(d+1)})(r-1)+\frac{d}{2^d(d+1)}}, s^{2d+2-2^{-r}}\bigr \}.\) 
\end{prop}
\begin{proof}[Proof of Proposition~\ref{prop:Counting}]
    By symmetry, without loss of generality, we can assume that \(A=[r-1]\). We first show the following claim.
    \begin{claim}\label{claim:auxi}
        For convex sets \(D_{i,j}\subseteq\mathbb{R}^{d}\) with \(i\in [d+1]\) and \(j\in [2]\) such that \(D_{i,1}\cap D_{i,2}=\emptyset\) for each \(i\in [d+1]\), there exists a \((d+1)\)-tuple \((j_{1},j_{2},\ldots,j_{d+1})\in [2]^{d+1}\) such that \(\bigcap_{i=1}^{d+1}D_{i,j_{i}}=\emptyset\).
    \end{claim}
    \begin{poc}
We prove it by induction on \(d\). The case of \(d=1\) is trivial. Assume that the conclusion holds for all values smaller than \(d+1\). Suppose that for any \((d+1)\)-tuple \((j_{1},j_{2},\ldots,j_{d+1})\in [2]^{d+1}\), we have \(\bigcap_{\ell=1}^{d+1}D_{\ell,j_{\ell}}\neq\emptyset\), where each \(D_{\ell,j_{\ell}}\in \mathbb{R}^d\). Note that for each \(i\in [d+1]\), we can find a hyperplane \(H_{i}\) that separates the disjoint pair of convex sets \(D_{i,1},D_{i,2}\) by~\cref{thm:SeparatingTheorem}.

We then show that for any \((j_{1},j_{2},\ldots,j_{d})\in [2]^{d}\), we have \((\bigcap_{i=1}^{d}D_{i,j_{i}})\cap H_{d+1}\neq\emptyset.\) To see this, we take points \(\bm{u}_{1}\in (\bigcap_{i=1}^{d}D_{i,j_{i}})\cap D_{d+1,1}\) and \(\bm{u}_{2}\in (\bigcap_{i=1}^{d}D_{i,j_{i}})\cap D_{d+1,2}\) respectively. Since \(H_{d+1}\) separates \(D_{d+1,1}\) and \(D_{d+1,2}\), we can see \(\textup{Conv}(\{\bm{u}_{1},\bm{u}_{2}\})\cap H_{d+1}\neq\emptyset\). This implies that 
\begin{equation}\label{equation:Intersection}
\bigg(\bigcap_{i=1}^{d}D_{i,j_{i}}\bigg)\cap H_{d+1}\neq\emptyset.
\end{equation}

Define \(D_{i,j}':=D_{i,j}\cap H_{d+1}\) for each \(i\in[d]\) and \(j\in [2]\). Since $H_{d+1}\cong \mathbb{R}^{d-1}$, and $D_{i,j}'\in H_{d+1}$ for each $D_{i,j}'$, we can apply the inductive hypothesis for the $(d-1)$-dimensional case to deduce that there exist \(j_{1},\ldots,j_{d}\) such that \(\bigcap_{i=1}^{d}D_{i,j_{i}}'=\emptyset,\) which is a contradiction to~\eqref{equation:Intersection}. 
\end{poc}
Now we define an \((r-1)\)-partite \((r-1)\)-uniform hypergraph \(\mathcal{G}=(V,E)\), where \(V=\bigcup_{i=1}^{r-1}V_{i}\), and \(V_{i}=\{C_{i,j}\}_{j\in [s]}\), that is, we regard each convex set \(C_{i,j}\) as a vertex in the hypergraph \(\mathcal{G}\). Furthermore, an \((r-1)\)-tuple \((C_{1,j_{1}},C_{2,j_{2}},\ldots,C_{r-1,j_{r-1}})\) forms an edge if and only if \(\bigcap_{i=1}^{r-1}C_{i,j_{i}}\neq\emptyset\). Then for \(A=[r-1]\), the number of \((r-1)\)-tuples \((i_{\ell})_{\ell\in A}\in [s]^{r-1}\) such that \(\bigcap_{\ell\in A}C_{\ell,i_{\ell}}\neq\emptyset\) is exactly \(|E(\mathcal{G})|\).

Next we focus on the upper bound for \(|E(\mathcal{G})|\). Noting that \(r-1\ge d+1\), for a \((d+1)\)-subset \(T=\{t_{1},t_{2},\ldots,t_{d+1}\}\) of \([r-1]\), we define the \((d+1)\)-partite projection hypergraph \(\mathcal{G}_{T}\) with vertex set \(\bigcup_{i\in T}V_{i}\), and a \((d+1)\)-tuple \(f=(C_{t_{1},j_{t_{1}}},\ldots,C_{t_{d+1},j_{t_{d+1}}})\) forms an edge in \(\mathcal{G}_{T}\) if and only if there exists some edge \(e\in E(\mathcal{G})\) such that \(f\subseteq e\).

By Claim~\ref{claim:auxi}, we have the following corollary.
\begin{claim}\label{claim:K2222}
    For any \((d+1)\)-subset \(T\subseteq [r-1]\), \(\mathcal{G}_{T}\) is \(K_{2,2,\ldots,2}^{(d+1)}\)-free.
\end{claim}
By~\cref{erdosBox}, for any \(T\), we have \(|E(\mathcal{G}_{T})|\le c_{\eqref{erdosBox}}s^{d+1-\frac{1}{2^{d}}}\), where \(c_{\eqref{erdosBox}}\) only depends on \(d\). If we divide \(V\) into \(\ceil{\frac{r-1}{d+1}}\) blocks, then we can immediately obtain the first upper bound of $|E(\mathcal{G})|$ as
\[
O\bigg(s^{\floor{\frac{r-1}{d+1}}\cdot (d+1-\frac{1}{2^{d}})+(r-1)-(d+1)\floor{\frac{r-1}{d+1}}}\bigg)=O\bigg(s^{\frac{r-1}{d+1}\cdot (d+1-\frac{1}{2^{d}})+d\cdot \frac{1}{2^d(d+1)}}\bigg)=O\bigg(s^{\big(1-\frac{1}{2^{d}(d+1)}\big)(r-1)+\frac{d}{2^{d}(d+1)}}\bigg).
\]

Moreover, we can obtain a better control when \(r\) is relatively larger than \(d\) via a slightly different viewpoint.
For any subset \(T=\{t_{1},\ldots,t_{d+1}\}\subseteq [r-1]\) of size \(d+1\) , we define 
\(\psi_{T}: [s]^{r-1}\rightarrow [s]^{d+1}\)  by
\[
\psi_{T}(x_1,x_2,\cdots,x_{r-1})=(x_{t_{1}},x_{t_2},\ldots,x_{t_{d+1}}).
\]
A \emph{$k$-dimensional hypercube} $\mathcal{H}_{k}\subseteq [s]^k$ is defined by
$$\mathcal{H}_{k}=\{(a_{1,i_1},a_{2,i_2},\cdots,a_{k,i_k})\subseteq [s]^{k}: i_j\in [2], j\in[k]\}.$$

We consider the following function \(F(d+1,r-1,s)\), which is defined to be the largest size of a subset \(S\subseteq [s]^{r-1}\) such that for any subset \(T\subseteq [r-1]\) of size \(d+1\), \(\psi_{T}(S)\) does not contain any \((d+1)\)-dimensional hypercube. It is not hard to see \(|E(\mathcal{G})|\le F(d+1,r-1,s)\).

\begin{claim}\label{upperboundBox}
    \(F(d+1,r-1,s)\le s^{2d+2-2^{-r}}\).
\end{claim}
\begin{poc}
    We prove this by induction on \(d+r\). Obviously we have \(F(1,r-1,s)=1\). Moreover, we have \(F(d+1,r-1,s)\le s^{r-1}<s^{2d+2-2^{-r}}\) when \(s\) is large enough and \(d+1=r-1\). Suppose that for any \(d',r'\) with \(d'+r'< d+r\), we have \(F(d'+1,r'-1,s)\le s^{2d'+2-2^{-r'}}.\) Let \(S\subseteq [s]^{r-1}\) be a subset such that \(|S|=F(d+1,r-1,s)\), and for any subset \(T\subseteq [r-1]\) of size \(d+1\), \(\psi_{T}(S)\) does not contain any \((d+1)\)-dimensional hypercube. We then define a mapping \(\gamma:[s]^{r-1}\rightarrow [s]^{r-2}\) by
\[
\gamma(x_{1},\ldots,x_{r-1}):=(x_{2},\ldots,x_{r-1}).
\]
Then the size of \(\gamma(S)\) can be upper bounded by \(F(d+1,r-2,s)\). For distinct $i,j\in [s]$, let 
\[
T_{i,j}:=\{X\in \gamma(S):\{i\}\times X\in S\ \textup{and}\  \{j\}\times X\in S\}. 
\]
Then we have \(|T_{i,j}|\le F(d,r-2,s)\). By double counting, we have
\begin{align*}
\binom{s}{2}\cdot F(d,r-2,s)
&\ge \sum\limits_{1\leq i<j\leq s}|T_{i,j}|\\
&= \sum\limits_{1\leq i<j\leq s}\sum\limits_{X\in T_{i,j}}\mathbbm{1}_{(\{i\}\times X\in S)\wedge (\{j\}\times X\in S)}\\
&= \sum\limits_{X\in\gamma(S)}\binom{|\gamma^{-1}(X)|}{2}\\
&\ge |\gamma(S)|\cdot\binom{|S|/|\gamma(S)|}{2}\\
&\ge \frac{|S|}{2}\cdot\Big(\frac{|S|}{F(d+1,r-2,s)}-1\Big),
\end{align*}
where the second inequality follows from the convexity.
 Therefore, we have
 \[
 \frac{F(d+1,r-1,s)^{2}}{F(d+1,r-2,s)}-F(d+1,r-1,s)-s(s-1)\cdot F(d,r-2,s)\le 0.
 \]
 Then by inductive hypothesis, we have
\begin{align*}
F(d+1,r-1,s)
&\le \frac{ F(d+1,r-2,s)}{2}
  +\sqrt{\frac{F(d+1,r-2,s)^{2}}{4}+s(s-1) F(d+1,r-2,s) F(d,r-2,s)}\\
&\le \frac{1}{2}\cdot s^{2(d+1)-2^{-(r-1)}}+\sqrt{\frac{1}{4}\cdot s^{4(d+1)-2^{-r+2}}+s^{4(d+1)-2^{-r+2}}}\\
&= \frac{1+\sqrt{5}}{2}\cdot s^{2(d+1)-2^{-r+1}}\\
&\le s^{2d+2-2^{-r}}.
\end{align*}
This finishes the proof.
\end{poc}
This completes the proof of Proposition~\ref{prop:Counting}.
\end{proof}

In the following result, let $q=q(d,s,r)$ be the value appearing in Proposition~\ref{prop:Counting}.

\begin{prop}\label{prop:ConvextPolyhedrons}
Let \(C_{1},\ldots,C_{r}\) be \(s\)-convex sets in \(\mathbb{R}^{d}\) with \(\bigcap_{j=1}^{r}C_{j}=\emptyset\), where for each \(i\in [r]\), $C_{i}=\bigcup_{k=1}^{s}C_{i,k}$ with $C_{i,j}\cap C_{i,j'}=\emptyset$ for all $j\ne j'$. Then there exist \(r\) sets \(K_{1},K_{2},\ldots,K_{r}\) where each \(K_{k}\) is the union of \(s\) convex polyhedrons with a total of at most \(rs(q+s)\) facets such that \(C_{i}\subseteq K_{i}\) for each \(i\in [r]\) and \(\bigcap_{i=1}^{r}K_{i}=\emptyset.\)
\end{prop}
\begin{proof}[Proof of Proposition~\ref{prop:ConvextPolyhedrons}]
    Our strategy is for each \(i=1,2,\ldots,r\), we iteratively find the union of \(s\) convex polyhedrons \(K_{i}\) which contains \(C_{i}\) while keeping the cumulative number of facets under tight control.

    First for a convex set \(C_{1,1}\), for each \(2\le j\le s\), since \(C_{1,1}\) is disjoint from \(C_{1,j}\), by~\cref{thm:SeparatingTheorem} we can find a halfspace \(\mathbb{H}_{1,j}\) to separate \(C_{1,1}\subseteq \mathbb{H}_{1,j}\) and \(C_{1,j}\subseteq (\mathbb{H}_{1,j})^{c}\). Furthermore by Proposition~\ref{prop:Counting}, for \(A_{1}=\{2,3,\ldots,r\}\), the number of \((r-1)\)-tuples \((i_{\ell})_{\ell\in [A_{1}]}\in [s]^{r-1}\) such that \(\bigcap_{\ell\in A_{1}}C_{\ell,i_{\ell}}\neq\emptyset\) is at most \(q\). Thus totally we need at most \((s-1)+q\) halfspaces whose intersection results in a convex polyhedron \(K_{1,1}\) such that \(C_{1,1}\subseteq K_{1,1}\), \(K_{1,1}\cap C_{1,j}=\emptyset\) for each \(2\le j\le s\), and \(K_{1,1}\cap (\bigcap_{\ell\in A_{1}}C_{\ell,i_{\ell}})=\emptyset\) for all those {\((r-1)\)-tuples \((i_{\ell})_{\ell\in [A_{1}]}\)} satisfying \(\bigcap_{\ell\in A_{1}}C_{\ell,i_{\ell}}\neq\emptyset\). Moreover, the number of facets of \(K_{1,1}\) is at most \(q+s-1\). We then replace the original \(C_{1,1}\) with \(K_{1,1}\).

    Running the same operation for at most \(rs\) times, we can find the desired \(K_{1},K_{2},\ldots,K_{r}\), where \(K_{i}=\bigcup_{j=1}^{s}K_{i,j}\). Moreover, the total number of facets is at most \(rs(q+s-1)\). This finishes the proof.

\end{proof}

We next derive the main theorem from Proposition~\ref{prop:ConvextPolyhedrons}, drawing inspiration from the proof of~\cite[Theorem~1.6]{2025arxivAlonSmo}.
\paragraph{Proof of~\cref{thm:GUP} via Proposition~\ref{prop:ConvextPolyhedrons}} 
Let \(\ell=rs(q+s)\). Let \(\mathcal{H}=(P,E)\), where \(S\subset P\) is a hyperedge (that is, \(S\in E\)) if and only if there exists a positive integer $m$ and a set of convex polyhedrons \(K_{1},\ldots,K_{m}\) with a total of at most \(\ell\) facets such that \(S=P\cap (\bigcup_{j=1}^{m}K_{j})\). By Lemma~\ref{lem:AS-VC}, the VC-dimension of \(\mathcal{H}\) is at most \(c_{\eqref{lem:AS-VC}}d\ell\log{\ell}\). Furthermore, by Lemma~\ref{lemma:RSS}, the largest size \(n\) of \(r\)-shattered set of \(\mathcal{H}\) is at most
\[
c_{\eqref{lemma:RSS}}dr^{2}(\log{r})\cdot\ell\log{\ell}\le  \min \bigl\{ c_{r,d}''\cdot s^{\bigl(1-\frac{1}{2^{d}(d+1)}\bigr)r+\frac{1}{2^d}}\cdot \log s,\ 
   c_{2}d^{2}r^{3}\log r\cdot s^{2d+3-2^{-r}}\cdot \log s\bigr\}
   \] for some constant \(c_{r,d}''\) depending on \(d,r\) and some absolute constant \(c_{2}>0\).

Then we show \(F_{r}(d,s,\ldots,s)\le n+1\). Let \(P\subseteq \mathbb{R}^{d}\) be a set of size \(n+1\). Then \(P\) cannot be \(r\)-shattered by \(\mathcal{H}\). That means there is some partition \(P=\bigsqcup_{i=1}^{r}P_{i}\) with the following property: there does not exist a collection of \(K_{1},\ldots,K_{r}\) each of which is the union of convex polyhedrons with a total of at most \(rs(s^{2d}+s)\) facets such that \(P_{i}\subseteq K_{i}\) for each \(i\in [r]\) and \(\bigcap_{i=1}^{r}K_{i}=\emptyset\). Now suppose that there are \(C_{1},\ldots,C_{r}\) being \(s\)-convex sets in \(\mathbb{R}^{d}\) with \(\bigcap_{i=1}^{r}C_{i}=\emptyset\) and \(P_{i}\subseteq C_{i}\). Then by Proposition~\ref{prop:ConvextPolyhedrons}, there exist \(r\) sets \(K_{1},K_{2},\ldots,K_{r}\) where each \(K_{k}\) is a union of \(s\) convex polyhedrons with a total of at most \(rs(s^{2d}+s)\) facets such that \(C_{i}\subseteq K_{i}\) for each \(i\in [r]\) and \(\bigcap_{i=1}^{r}K_{i}=\emptyset,\) which is a contradiction. This finishes the proof.
\end{proof}

\section{Concluding remarks}\label{sec:remarks}
In this paper we establish two negative answers to the questions of Alon and Smorodinsky concerning Tverberg-type intersections of unions of convex sets. We prove that for all \(r\ge 2\), \(s\ge 1\), and \(d\ge 2r-2\),
\[
   f_r(d,s,\ldots,s) > s^r,
\]
which matches their general upper bound up to a logarithmic factor. Our construction combines a planar ``ring configuration'' with a gluing scheme on a high-dimensional torus. We further introduce the disjoint-union variant \(F_r(d,s_1,\ldots,s_r)\), revealing that the disjointness drastically alters the quantitative behaviour: in particular, \(F_2(2,s,s)=O(s\log s)\), and for general parameters we obtain upper bounds via a novel connection to hypergraph Tur\'{a}n theory.

As we have mentioned in~\cref{subsection:Contribution}, we can further improve the lower bound as
\[
f_{r}(d,s,\ldots,s)>(d-2r+4)s^{r},
\]
for any \(s\ge 1\), \(r\ge 2\) and \(d\ge 2r-2\). This constitutes a further step toward closing the gap with the upper bound \(O_{r}(ds^{r}\log{s})\) in~\cref{thm:AS-general} obtained by Alon and Smorodinsky when both of \(s\) and \(d\) are large. Here we describe the construction to show \(f_{2}(d,s,s)>ds^{2}\) in details and the general lower bound follows from the same ideas as that in~\cref{section:Generalization High}. 

To achieve this, in $\mathbb{R}^d$, the original selection of $s^2$ points in $\mathbb{R}^2$ can be modified by replacing each circular arc with a spherical surface $\mathbb{S}^{d-1}$ in $\mathbb{R}^d$, and subsequently mapping each point on it to the $d$ vertices of a small $(d-1)$-dimensional regular simplex on this $\mathbb{S}^{d-1}$.

Formally, the construction of the $d \cdot s^2$ point set proceeds as follows.  
Following the approach used for selecting $s^2$ points in the proof of~\cref{thm:quadraticLB}, we work in $\mathbb{R}^d$ and select $s$ spherical surfaces $\mathbb{S}^{d-1}$, denoted $\mathcal{S}_k\cong\mathbb{S}^{d-1}$ for $k \in [s]$, each centered at $\bm{y}'_k = (\bm{y}_k, 0, \dots, 0)\in\mathbb{R}^{d}$ with sufficiently large radius \(R\) (\(R=M-1\) satisfies a similar condition as that in~\eqref{eq:choiceofM}), where \(\bm{y}_{k}\in\mathbb{R}^{2}\) is defined in~\eqref{eq:y_k}.
For each \(i,j\in [s]\), we also denote $\bm{z}'_{i,j}=(\bm{z}_{i,j},0,\dots,0)\in \mathbb{R}^d$, where \(\bm{z}_{i,j}\) is defined in~\eqref{eq:PPP}. Moreover, for each \(i,j\in [s]\), we then define a hyperplane $\mathbb{H}_{i,j}$ intersecting $\mathcal{S}_i$ (not only at \(\bm{z}'_{i,j}\)) such that the vector $\bm{z}'_{i,j} - \bm{y}'_i$ is orthogonal to $\mathbb{H}_{i,j}$, with the distance $\textup{dist}(\bm{z}'_{i,j}, \mathbb{H}_{i,j}) = \delta$, where $\delta$ is chosen sufficiently small relative to $\|\bm{z}'_{i,j}-\bm{z}'_{i,j+1}\|$ for all $i, j \in [s]$.

For each \(i,j\in [s]\), on the intersection $\mathbb{H}_{i,j} \cap \mathcal{S}_i$, we select $d$ points forming a $(d-1)$-dimensional regular simplex, which consists of the points $\bm{z}_{i,j,\ell}$ for $\ell \in [d]$.  
Then we define the set $P$ as  
\[
P := \{ \bm{z}_{i,j,\ell} : i \in [s],\ j \in [s],\ \ell \in [d] \}.
\]
Note that \(|P| = d \cdot s^2\). Now consider an arbitrary bipartition $P = P_1 \sqcup P_2$.  
For each row index $i \in [s]$, define the row convex container  
\[
C_i := \operatorname{Conv} \{ \bm{z}_{i,a,b} \in P_1 : a \in [s],\ b \in [d] \},
\]  
and for each column index $j \in [s]$, define the column convex container  
\[
D_j := \operatorname{Conv} \{ \bm{z}_{a,j,b} \in P_2 : a \in [s],\ b \in [d] \}.
\]  
Finally, define the union sets \(
C := \bigcup_{i=1}^s C_i\) and \(D := \bigcup_{j=1}^s D_j.
\)  
By construction, $C$ and $D$ are $s$-convex sets. Furthermore, it is immediate from the definitions that $P_1 \subseteq C$ and $P_2 \subseteq D$. The proof that $C \cap D = \emptyset$ requires more complicated analysis but follows an argument similar to that in Proposition~\ref{claim:DDDDDD}, here we give a brief sketch of the proof.  
For each \(i,j\in [s]\), define that 
\[
A_i := \operatorname{Conv} \{ \bm{z}_{i,a,b} : a \in [s],\ b \in [d] \},
\]  
\[
B_j := \operatorname{Conv} \{ \bm{z}_{a,j,b} : a \in [s],\ b \in [d] \}.
\]  
We have the following key observations:
\begin{itemize}
    \item The points $\{\bm{z}_{i,a,b}\}$ in $A_i$ and $\{\bm{z}_{a,j,b}\}$ in $B_j$ are in convex position, respectively.
    \item $A_i \cap B_j = \operatorname{Conv} \{ \bm{z}_{i,j,\ell} : \ell \in [d] \}$ since $A_i$ and $B_j$ lie in distinct closed half-spaces defined by $\mathbb{H}_{i,j}$.
\end{itemize}

 Since the set $\{ \bm{z}_{i,j,\ell} : \ell \in [d] \}$ constitutes the vertices of a $(d-1)$-dimensional regular simplex, for any partition $Z_{i,j} = Q_1 \sqcup Q_2$ we have  
\[
\operatorname{Conv}\{ \bm{z}_{i,j,\ell} \in Q_1,\ \ell \in [d] \} \cap \operatorname{Conv}\{ \bm{z}_{i,j,\ell} \in Q_2,\ \ell \in [d] \} = \emptyset.
\]  
Consequently, $C_i \cap D_j = \emptyset$ for all $i, j \in [s]$.

We show an example for interested readers in~\cref{fig:32}. 

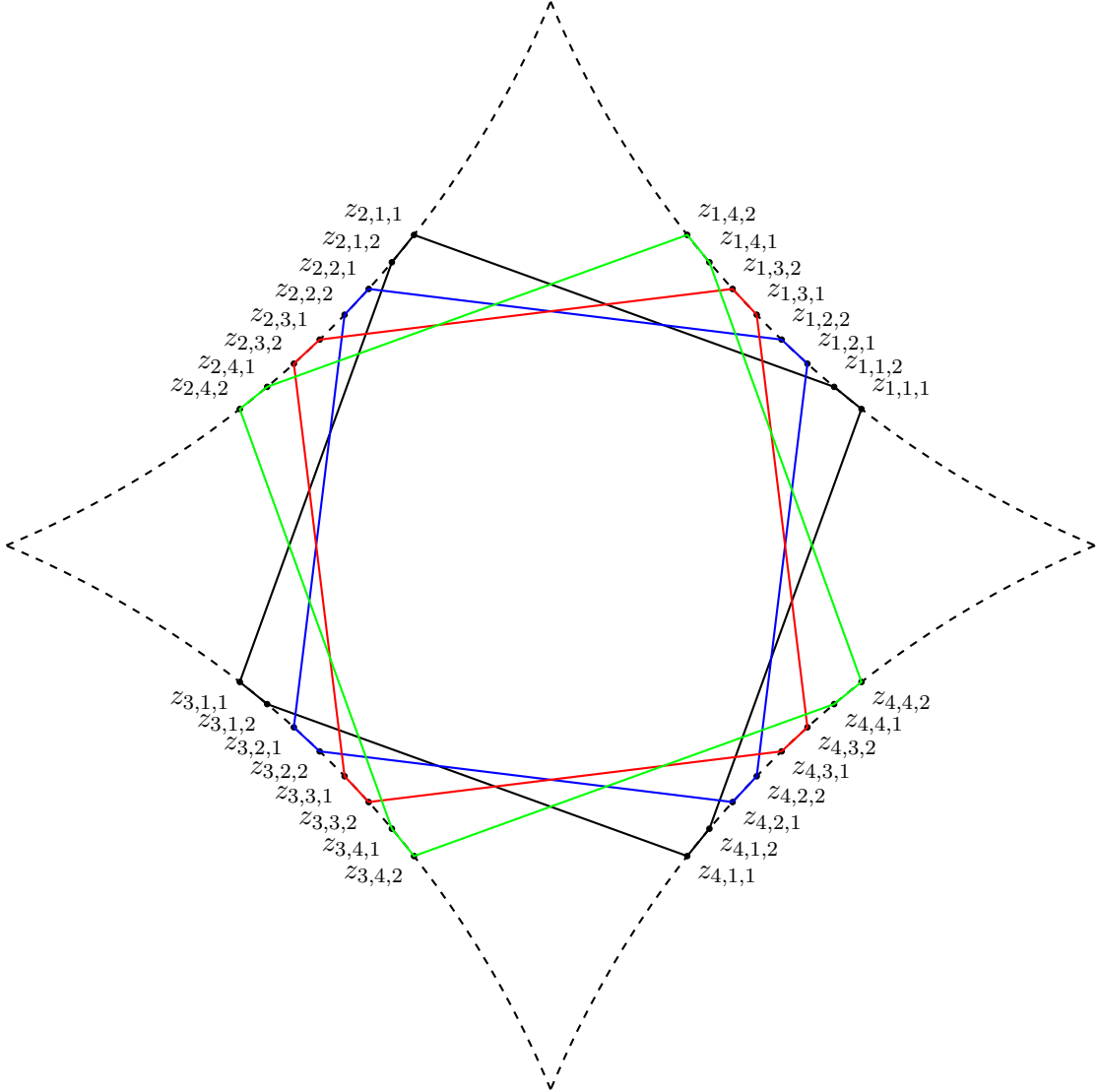
\begin{figure}[H]
    \centering
\begin{tikzpicture}[scale=1.05]
 \draw [thick,dashed] (0,7) arc(202.62:247.38:13);
 \draw [thick,dashed] (-7,0) arc(-67.38:-22.62:13);
 \draw [thick,dashed] (0,-7) arc(22.62:67.38:13);
 \draw [thick,dashed] (7,0) arc(112.62:157.38:13);

 \coordinate (z_{1,1,1}) at (3.996,1.756);
 \filldraw[black] (z_{1,1,1}) circle (1pt) node [above right]
{$z_{1,1,1}$};
 \coordinate (z_{1,1,2}) at (3.644,2.041);
 \filldraw[black] (z_{1,1,2}) circle (1pt) node [above right]
{$z_{1,1,2}$};
 \coordinate (z_{1,2,1}) at (3.301,2.339);
 \filldraw[black] (z_{1,2,1}) circle (1pt) node [above right]
{$z_{1,2,1}$};
 \coordinate (z_{1,2,2}) at (2.969,2.649);
 \filldraw[black] (z_{1,2,2}) circle (1pt) node [above right]
{$z_{1,2,2}$};
 \coordinate (z_{1,3,1}) at (2.649,2.969);
 \filldraw[black] (z_{1,3,1}) circle (1pt) node [above right]
{$z_{1,3,1}$};
 \coordinate (z_{1,3,2}) at (2.339,3.301);
 \filldraw[black] (z_{1,3,2}) circle (1pt) node [above right]
{$z_{1,3,2}$};
 \coordinate (z_{1,4,1}) at (2.041,3.644);
 \filldraw[black] (z_{1,4,1}) circle (1pt) node [above right]
{$z_{1,4,1}$};
 \coordinate (z_{1,4,2}) at (1.756,3.996);
 \filldraw[black] (z_{1,4,2}) circle (1pt) node [above right]
{$z_{1,4,2}$};
 \coordinate (z_{2,1,1}) at (-1.756,3.996);
 \filldraw[black] (z_{2,1,1}) circle (1pt) node [above left]
{$z_{2,1,1}$};
 \coordinate (z_{2,1,2}) at (-2.041,3.644);
 \filldraw[black] (z_{2,1,2}) circle (1pt) node [above left]
{$z_{2,1,2}$};
 \coordinate (z_{2,2,1}) at (-2.339,3.301);
 \filldraw[black] (z_{2,2,1}) circle (1pt) node [above left]
{$z_{2,2,1}$};
 \coordinate (z_{2,2,2}) at (-2.649,2.969);
 \filldraw[black] (z_{2,2,2}) circle (1pt) node [above left]
{$z_{2,2,2}$};
 \coordinate (z_{2,3,1}) at (-2.969,2.649);
 \filldraw[black] (z_{2,3,1}) circle (1pt) node [above left]
{$z_{2,3,1}$};
 \coordinate (z_{2,3,2}) at (-3.301,2.339);
 \filldraw[black] (z_{2,3,2}) circle (1pt) node [above left]
{$z_{2,3,2}$};
 \coordinate (z_{2,4,1}) at (-3.644,2.041);
 \filldraw[black] (z_{2,4,1}) circle (1pt) node [above left]
{$z_{2,4,1}$};
 \coordinate (z_{2,4,2}) at (-3.996,1.756);
 \filldraw[black] (z_{2,4,2}) circle (1pt) node [above left]
{$z_{2,4,2}$};
 \coordinate (z_{3,1,1}) at (-3.996,-1.756);
 \filldraw[black] (z_{3,1,1}) circle (1pt) node [below left]
{$z_{3,1,1}$};
 \coordinate (z_{3,1,2}) at (-3.644,-2.041);
 \filldraw[black] (z_{3,1,2}) circle (1pt) node [below left]
{$z_{3,1,2}$};
 \coordinate (z_{3,2,1}) at (-3.301,-2.339);
 \filldraw[black] (z_{3,2,1}) circle (1pt) node [below left]
{$z_{3,2,1}$};
 \coordinate (z_{3,2,2}) at (-2.969,-2.649);
 \filldraw[black] (z_{3,2,2}) circle (1pt) node [below left]
{$z_{3,2,2}$};
 \coordinate (z_{3,3,1}) at (-2.649,-2.969);
 \filldraw[black] (z_{3,3,1}) circle (1pt) node [below left]
{$z_{3,3,1}$};
 \coordinate (z_{3,3,2}) at (-2.339,-3.301);
 \filldraw[black] (z_{3,3,2}) circle (1pt) node [below left]
{$z_{3,3,2}$};
 \coordinate (z_{3,4,1}) at (-2.041,-3.644);
 \filldraw[black] (z_{3,4,1}) circle (1pt) node [below left]
{$z_{3,4,1}$};
 \coordinate (z_{3,4,2}) at (-1.756,-3.996);
 \filldraw[black] (z_{3,4,2}) circle (1pt) node [below left]
{$z_{3,4,2}$};
 \coordinate (z_{4,1,1}) at (1.756,-3.996);
 \filldraw[black] (z_{4,1,1}) circle (1pt) node [below right]
{$z_{4,1,1}$};
 \coordinate (z_{4,1,2}) at (2.041,-3.644);
 \filldraw[black] (z_{4,1,2}) circle (1pt) node [below right]
{$z_{4,1,2}$};
 \coordinate (z_{4,2,1}) at (2.339,-3.301);
 \filldraw[black] (z_{4,2,1}) circle (1pt) node [below right]
{$z_{4,2,1}$};
 \coordinate (z_{4,2,2}) at (2.649,-2.969);
 \filldraw[black] (z_{4,2,2}) circle (1pt) node [below right]
{$z_{4,2,2}$};
 \coordinate (z_{4,3,1}) at (2.969,-2.649);
 \filldraw[black] (z_{4,3,1}) circle (1pt) node [below right]
{$z_{4,3,1}$};
 \coordinate (z_{4,3,2}) at (3.301,-2.339);
 \filldraw[black] (z_{4,3,2}) circle (1pt) node [below right]
{$z_{4,3,2}$};
 \coordinate (z_{4,4,1}) at (3.644,-2.041);
 \filldraw[black] (z_{4,4,1}) circle (1pt) node [below right]
{$z_{4,4,1}$};
 \coordinate (z_{4,4,2}) at (3.996,-1.756);
 \filldraw[black] (z_{4,4,2}) circle (1pt) node [below right]
{$z_{4,4,2}$};
\draw[black,thick](z_{1,1,1})--(z_{1,1,2})--(z_{2,1,1})--(z_{2,1,2})--(z_{3,1,1})--(z_{3,1,2})--(z_{4,1,1})--(z_{4,1,2})--(z_{1,1,1});
\draw[blue,thick](z_{1,2,1})--(z_{1,2,2})--(z_{2,2,1})--(z_{2,2,2})--(z_{3,2,1})--(z_{3,2,2})--(z_{4,2,1})--(z_{4,2,2})--(z_{1,2,1});
\draw[red,thick](z_{1,3,1})--(z_{1,3,2})--(z_{2,3,1})--(z_{2,3,2})--(z_{3,3,1})--(z_{3,3,2})--(z_{4,3,1})--(z_{4,3,2})--(z_{1,3,1});
\draw[green,thick](z_{1,4,1})--(z_{1,4,2})--(z_{2,4,1})--(z_{2,4,2})--(z_{3,4,1})--(z_{3,4,2})--(z_{4,4,1})--(z_{4,4,2})--(z_{1,4,1});

\end{tikzpicture}
\caption{\(f_{2}(2,4,4)>32\)}
\label{fig:32}
\end{figure}

For the original function \(f_r(d,s,\ldots,s)\), many questions remain open. As Kalai emphasized in his blog~\cite{2025KalaiBolg}, it is a very promising direction in discrete geometry, and our results only mark the beginning. Even in the planar symmetric case, it is unclear whether
\[
   f_2(2,s,s)=\Theta(s^2)
\]
holds. The logarithmic factor in the upper bound of Alon and Smorodinsky arises from a VC-dimension estimate, and it is plausible that a proof avoiding this black-box machinery could remove the extra \(\log s\) term. Developing such a direct geometric argument would be highly desirable.

Turning to the disjoint-union model \(F_r(d,s_1,\ldots,s_r)\), we find it equally fascinating: disjointness appears to play an unexpectedly powerful role in shaping the extremal behavior. This is exemplified in Claim~\ref{claim:auxi}, where we leverage the disjointness property to demonstrate that the projection hypergraph is free of \(K_{2,2,\ldots,2}^{(d+1)}\). This key observation bridges our problem to the field of hypergraph Tur\'{a}n theory. Regarding the function \(F(d,r,s)\) in Claim~\ref{upperboundBox}, what we encounter appears to define a new class of Tur\'{a}n-type problems, for which we have not found any existing results providing superior upper bounds, which might be of independent interest.

Following the spirit of Alon and Smorodinsky, we raise an analogous question, though it may well have a negative answer.

\begin{ques}
{Is \(F_r(d,s,\ldots,s)\) polynomially bounded in \(r,d,s\)?}
\end{ques}
\cref{thm:AS-general} and \cref{thm:GeneralUpperbound} show that, if one of \(d\) and \(r\) is fixed, then the polynomial upper bound holds.

\section*{Acknowledgement}
The authors would like to thank Noga Alon, Gil Kalai and Shakhar Smorodinsky for their helpful comments. In particular, they thank Shakhar Smorodinsky for bringing to their attention the result of~\cite{1990DM}, which significantly simplifies the proof of~\cref{thm:VariantUBPlanar}.

As a participant in the ECOPRO Summer Student Research Program, Wenchong Chen acknowledges the organizers, especially Hong Liu, and the members of IBS for their support and hospitality. Zixiang Xu would like to thank Hong Liu and Chong Shangguan for their encouragement over the past few years to pursue topics in combinatorial convexity, in particular, he gratefully acknowledges Chong Shangguan’s gift of the print edition of~\cite{2021CombConvex} in 2023. While working on this project in September, Zixiang Xu thanks Jiangdong Ai, Fankang He, Yulai Ma, Yongtang Shi and Shuaichao Wang for their hospitality at Nankai University; Jie Han for his hospitality at the Beijing Institute of Technology; and Luyi Li, Ke Ye, and Qiang Zhou for their hospitality at the Chinese Academy of Sciences.

\bibliographystyle{abbrv}
\bibliography{Tverberg}

\end{document}